\documentclass[10pt]{article}
\usepackage{amsmath,amscd}
\usepackage{amssymb,latexsym,amsthm}
\usepackage{color,palatino}
\usepackage[utf8]{inputenc}
\usepackage[T1]{fontenc}
\usepackage{csquotes}
\usepackage{latexsym}
\usepackage{multirow}
\usepackage{lscape}
\usepackage{enumerate}
\usepackage{fancyhdr}
\usepackage{pb-diagram}
\usepackage{amsfonts}
\usepackage{hyperref}
\hypersetup{
  colorlinks=true,
  linkcolor=blue,
  citecolor=black,
  urlcolor=cyan
}

\newtheorem{teorema}{Theorem}[section]
\newtheorem{proposicion}[teorema]{Proposition}
\newtheorem{lema}[teorema]{Lemma}

\newtheorem{corolario}[teorema]{Corollary}
\newtheorem{remark}[teorema]{Note}
\newtheorem{ejemplo}[teorema]{Example}

\newtheorem{definicion}[teorema]{Definition}

\newtheorem{observacion}[teorema]{Remark}

\newcommand{\pqnum}[1]{\left[#1\right]_{p,q}}
\newcommand{\Hpq}{H_{p,q}}
\newcommand{\K}{\Bbbk}

\newcommand{\Z}{\mathbb{Z}}
\newcommand{\N}{\mathbb{N}}

\makeatletter
\def\@roman#1{\romannumeral #1}
\makeatother

\title{\textbf{The PI Property in Algebras of Polynomial Type}}
\author{James Gómez and Claudia Gallego} 
\date{}
\begin{document}
\makeatletter
\def\@roman#1{\romannumeral #1}
\makeatother
\maketitle
\begin{abstract}
\noindent
In this article, we study the PI property for several families of
noncommutative algebras of polynomial type. 
Specifically, we review
criteria for the PI property in double Ore extensions, two-parameter quantum
Heisenberg algebras, two-parameter quantum matrix algebras, the algebra
$U_q^+(B_2)$, multiparametric quantum Weyl algebras, biquadratic algebras
with three generators, Noetherian Down--Up algebras, and the recently
introduced algebras $B_q(f)$. In several cases, we include detailed proofs of
known results, provide proofs of some identities used in the literature, and
present alternative proofs of results characterizing the PI property for some
of these algebras. For Noetherian Down--Up algebras, we highlight the
relationship between the PI property, finiteness over the center, and the FBN
property. Finally, for the algebras $B_q(f)$, we  prove that they
admit a PBW basis and show that the PI property can be controlled in terms of
the support of the polynomial $f$.
\bigskip

\bigskip

\end{abstract}
\section{Introduction}
\noindent
The earliest research on algebras satisfying polynomial identities (PI algebras) was developed by Dehn in 1922, motivated by determining what types of theorems are valid in non-Archimedean geometries (see \cite{dehn}). On the other hand, in 1937, Wagner in his article about the foundations of projective geometry established polynomial identities for matrix algebras and quaternion algebras, see \cite[Chapter 2]{wagner}; shortly thereafter, in 1947,  Jacobson and  Kaplansky laid the first foundations of what is now known as PI theory, see \cite{jacobson, kaplansky}. Furthermore, Pascaud and Valette in \cite{pasca}, Cauchon in \cite{cauchon}, and Damiano and Shapiro in \cite{damiano} provide necessary and sufficient conditions for an Ore extension  $R[x; \sigma, \delta]$ to satisfy a polynomial identity. In recent years, the study of algebras satisfying a polynomial identity (the PI property) has gained renewed interest in noncommutative algebra. This property serves, for instance, as a key tool for cla\-ssi\-fying the simple modules of an algebra, establishing relations between an algebra and its center, and studying finite generation over central subalgebras. Moreover, in recent investigations on the Ozone group, as defined by Chan, Gaddis, Won, and Zhang, the PI property is used as a key hypothesis for the algebras under consideration; see \cite{BeraMukherjee2023, ChanGaddisWonZhang2025}.

In this article, we present in detail results on the PI property for several
algebras of polynomial type, including some recently introduced ones.

We provide detailed proofs of known results on the PI \-pro\-per\-ty, including
proofs of some identities used in the arguments, for double Ore extensions
(Section~\ref{oredobles}), two-parameter Heisenberg algebras
(Section~\ref{Heisenbergdedosparametros}), the algebra $U_q^+(B_2)$
(Section~\ref{álgebraU_q+}), and Noetherian Down--Up algebras
(Section~\ref{DownUpalgebras}). For these classes of algebras, we emphasize
a characterization of the PI property in terms of roots of unity, finiteness
over the center, and the FBN property. 

We also establish alternative proofs to those available for two-parameter
quantum matrix algebras (Section~\ref{matrices2p}), multiparametric quantum
Weyl algebras (Section~\ref{Wely cuanticas}), and biquadratic algebras with
three generators (Section~\ref{bicuadráticas}).

In Section \ref{$B_q(f)$}, we consider the recently introduced family of algebras
$B_q(f)$. We prove that these algebras admit a PBW basis and show
that the PI property of $B_q(f)$ can be controlled in terms of the support of
the defining polynomial $f$. When $q$ has finite multiplicative order $n$, the
condition
\[
\operatorname{supp}(f)\subseteq n\mathbb{N}
\]
guarantees that $B_q(f)$ satisfies a polynomial identity. Thus, the
divisibility of the exponents appearing in $f$ by the order of $q$ provides a
natural sufficient condition for the algebra to satisfy the PI property. This
phenomenon is illustrated by means of some particular examples. In addition,
we present in detail the results proposed in \cite{GaddisYee2025} concerning
the PI property for these algebras, showing that the condition that $q$ be a
root of unity is necessary for $B_q(f)$ to be a PI algebra. Finally, in the
positive-characteristic setting, we analyze how this hypothesis influences the
description of some central elements of $B_q(f)$. 

Throughout this article, $\K$ denotes an algebraically closed field of
arbitrary characteristic, except for the results treated in Sections \ref{Heisenbergdedosparametros}, \ref{DownUpalgebras}
and \ref{$B_q(f)$}. In Sections \ref{Heisenbergdedosparametros}
and \ref{DownUpalgebras}, which are devoted to the two-parameter quantum
Heisenberg algebra $H_{p,q}$ and to Down-Up algebras, respectively, the
base field is assumed to have characteristic zero. On the other hand, in
Section \ref{$B_q(f)$}, devoted to the algebras $B_q(f)$, the role of the
characteristic is made explicit; in particular, some results are discussed
separately according to whether the base field has characteristic zero or
positive characteristic.

Let $A$ be an arbitrary ring. A polynomial identity for $A$ is a polynomial
\[
f(x_1,\ldots,x_n)\in \mathbb{Z}\langle x_1,\ldots,x_n\rangle
\]
such that
\[
f(a_1,\ldots,a_n)=0
\]
for all $a_1,\ldots,a_n\in A$. Following the convention in
\cite[Section~13.1]{mcconnell}, a ring $A$ is called a
PI ring if it satisfies a monic polynomial identity in
$\mathbb{Z}\langle x_1,\ldots,x_n\rangle$.

More generally, let $R$ be a commutative ring and let $A$ be an
$R$-algebra. A polynomial identity for $A$ over $R$ is a polynomial
\[
f(x_1,\ldots,x_n)\in R\langle x_1,\ldots,x_n\rangle
\]
such that
\[
f(a_1,\ldots,a_n)=0
\]
for all $a_1,\ldots,a_n\in A$. In this paper, an $R$-algebra $A$ will be
called a PI algebra if its underlying ring is a PI ring in the preceding
sense.

When the base ring is a field $\K$, this convention agrees with the usual
one: a $\K$-algebra is PI if and only if it satisfies a nonzero polynomial
identity in $\K\langle x_1,\ldots,x_n\rangle$. Indeed, if such an identity
has a nonzero coefficient among its highest-degree monomials, then one can
multiply the polynomial by the inverse of that coefficient and obtain a
monic polynomial identity.


\section{Double Ore extensions}\label{oredobles}
Double Ore extensions were defined in \cite{Zhang} as a generalization of Ore extensions. Using double Ore extensions, 26 families of Artin-Schelter regular algebras of global dimension 4 were obtained (see \cite{Zhang2}). Double Ore extensions are constructed from a ring $A$ by the simultaneous adjunction to $A$ of two generators $x_1$ and $x_2$.

\begin{definicion} [{\cite[Definition 1.3]{Zhang}}]\label{defore2} Let $A$ be an algebra and $B$ another algebra containing $A$  as a subring:
\begin{itemize}
\item[\rm (i)] It is said that $B$ is  a \textbf{right double Ore extension of} $A$ if the following conditions hold:
\begin{enumerate}
\item[\rm (ia)] The ring $B$ is generated by $A$ and two new variables $x_1$ and $x_2$.
\item[\rm (iia)] The variables $x_1$ and $x_2$ satisfy the relation
\begin{equation}\label{R1}
x_2x_1= p_{12}x_1x_2 + p_{11}x_1^2 + \tau_1x_1 + \tau_2 x_2 + \tau_0,
\end{equation}
where $p_{12}$, $p_{11}\in \K$ and $\tau_1$,  $\tau_2$, $\tau_0 \in A$.
\item[\rm (iiia)] As a left $A$-module, the $\K$-algebra $B$ is free with basis
\begin{align*} 
\{x_1^{\alpha_1} x_2^{\alpha_2} \mid \alpha_1,  \alpha_2 \geq 0\}.
\end{align*}
\item[\rm (iva)] \begin{equation}\label{4ore2}
x_1A + x_2A \subseteq Ax_1 + Ax_2 + A.
\end{equation}
\end{enumerate}
\item[\rm (ii)] Analogously, it is said that $B$ is a \textbf{left double Ore extension of} $A$ if the following conditions hold:
\begin{enumerate}
\item[\rm (ib)] The ring $B$ is generated by $A$ and two new variables $x_1$  and $x_2$.
\item[\rm (iib)] These variables $x_1$ and $x_2$ satisfy the relation
\begin{equation}\label{L1}
x_1x_2= p'_{12}x_2x_1 + p'_{11}x_1^2 + x_1\tau'_1 + x_2\tau'_2 + \tau'_0,
\end{equation}
where $p'_{12}$, $p'_{11} \in \K$
and $\tau'_1$, $\tau'_2$, $\tau'_0 \in A$.
\item[\rm (iiib)] As a right $A$-module the ring $B$ is free with basis \begin{align*} 
\{x_2^{\alpha_1} x_1^{\alpha_2} \mid \alpha_1, \alpha_2 \geq 0\}.
 \end{align*}
\item[\rm (ivb)] 
$Ax_{1} + Ax_{2} \subseteq x_{1}A + x_{2}A + A$.
\end{enumerate}
\item[\rm (iii)] If $B$ is both a left and a right double Ore extension of $A$ it is said that $B$ is \textbf{a double Ore extension of} $A$ with the same set of generators $\{x_1,x_2\}$.
\end{itemize}
\end{definicion}
In the graded case, all relations of $B$ are required to be homogeneous with $\deg(x_1)>0$ and $\deg(x_2) > 0$. We denote by $P$ the set of scalars $\{p_{12},p_{11}\}$, while $\tau$ denotes the set $\{\tau_1,\tau_2,\tau_0\}$. The sets $P$ and $\tau$ are called the \textbf{parameter} and the \textbf{tail}, respectively.

For computational purposes, it will be considerably more useful to have an explicit description of the definition of a double Ore extension, specifically with regard to the identities that establish the multiplication relations. To this end, condition (\ref{4ore2}) is rewritten in the following form:
\begin{equation}\label{R2}
\left(\begin{smallmatrix}
x_1 \\
x_2 \\
\end{smallmatrix}\right)r
 := \left(\begin{smallmatrix}
x_1r \\
x_2r \\
\end{smallmatrix}\right)= \left(\begin{smallmatrix}
\sigma_{11}(r) && \sigma_{12}(r) \\
\sigma_{21}(r) && \sigma_{22}(r) \\
\end{smallmatrix}\right)\left(\begin{smallmatrix}
x_1 \\
x_2 \\
\end{smallmatrix}\right) + \left(\begin{smallmatrix}
\delta_1(r) \\
\delta_2(r) \\
\end{smallmatrix}\right), 
\end{equation}

for all $r \in A$.
Writing $\sigma(r) = \left(\begin{smallmatrix}
\sigma_{11}(r) && \sigma_{12}(r) \\
\sigma_{21}(r) && \sigma_{22}(r) \\
\end{smallmatrix}\right)$ and $\delta(r) = \left(\begin{smallmatrix}
\delta_{1}(r)  \\
\delta_{2}(r)  \\
\end{smallmatrix}\right)$, such $\sigma$ turns out to be a $\K$-linear map from $A$ to $M_2(A)$, where $M_2(A)$ denotes the algebras of $2\times 2$ matrices over $A$, while $\delta$ is a $\K$-linear morphism from $A$ to the column $A$-module $A^{\oplus^{2}}:= \left(\begin{smallmatrix}
A \\
A  \\
\end{smallmatrix}\right)$. Thus, equation (\ref{R2}) can also be written in the following form:
$$\left(\begin{smallmatrix}
x_1\\
x_2 \\
\end{smallmatrix}\right)r = \sigma(r)\left(\begin{smallmatrix}
x_1\\
x_2\\
\end{smallmatrix}\right) + \delta(r).$$
This equation is a natural generalization of one of the conditions in the definition of an Ore extension.
\begin{definicion} Suppose that $\sigma : A \rightarrow M_2(A)$ is an algebra morphism, with $\sigma(rs)= \sigma(r)\sigma(s)$ for all $r,s \in A$. The $\K$-linear function $\delta: A \to A^{\oplus^2}$ is called a $\sigma$-\textbf{derivation} if:  $$\delta(rs)=\sigma(r)\delta(s) + \delta(r)s$$
for all $r$, $s \in A$.
\end{definicion}

\begin{remark}
\begin{itemize}
\item[\rm (1)] By virtue of what was established in (\ref{R2}), the morphisms $\sigma$
 and $\delta$
are completely determined. Under these conditions, a double Ore extension of $A$ is also denoted by $B=A_P[x_1,x_2; \sigma, \delta,\tau]$. With this notation it is understood that one is working with a right double Ore extension (although such a $B$ may also denote a left double Ore extension).
\item[\rm (2)] If $\delta$ is the zero map and $\tau=\emptyset$  then the right double Ore extension is denoted by $A_P[x_1,x_2;\sigma]$  and is called a \textbf{trimmed right double Ore extension}.
\item[\rm (3)] If $B=A_P[x_1,x_2;\sigma,\delta,\tau]$  denotes a (right-left) double Ore extension, then the collection 
$\{P, \sigma, \delta, \tau\}$ is called the \textbf{DE-data} of $A$.
\end{itemize}
\end{remark}

\begin{lema}[{\cite[Lemma 1.7]{Zhang}}] \label{lema1.7}
Let $B = A_P[x_1,x_2;\sigma, \delta, \tau]$ be a right double Ore extension of $A$, with $\{ \sigma,\delta \}$ those morphisms determined by (\ref{R2}). The following holds 
\begin{enumerate}
\item[\rm (i)] $\sigma: A \rightarrow M_2(A)$ is an algebra morphism.
\item[\rm (ii)] $\delta: A \rightarrow A^{\oplus^2}$ is a $\sigma$-derivation.
\end{enumerate}  
\end{lema}

\begin{ejemplo}
A trivial example of a right double Ore extension is obtained by taking $A=\K$. Specifically, if $B$ is a right double Ore extension of $\K$, then the following $\K$-isomorphism holds:
\[
B \cong \bigoplus_{n_1,n_2 \ge 0} \K\, x_1^{n_1} x_2^{n_2}
\]
Moreover, there is an isomorphism of $\K$-algebras
\[
\K\langle x_1, x_2\rangle / \langle r\rangle,
\]
where $r$ is the following relation:
\[
x_2 x_1 = p_{12} x_1 x_2 + p_{11} x_1^2 + a_1 x_1 + a_2 x_2 + a_3
\]
for certain $p_{12}, p_{11}, a_1, a_2, a_3 \in \K$ depending on $B$. The data associated with the double Ore extension can be identified explicitly
\[
\sigma = \begin{pmatrix} \mathrm{id}_{\K} & 0 \\ 0 & \mathrm{id}_{\K} \end{pmatrix}, 
\quad \delta = 0, 
\quad P = \{ p_{12}, p_{11} \}, 
\quad \tau = \{ a_1, a_2, a_3 \}.
\] If $p_{12} = 0$, then $B$ is a right double Ore extension, but not a double Ore extension, since condition (iib) in the Definition \ref{defore2} is not satisfied.  On the contrary, if $p_{12}\neq 0$, such $B$ turns out to be a double Ore extension.
\end{ejemplo}

The following are some notable examples of double Ore extensions introduced in \cite{Zhang}.
\begin{ejemplo}
For $A=\K[t]$, it is possible to classify all graded and connected right double extensions $B$ of $A$, considering $\deg(t)=\deg(x_1)=\deg(x_2)=1$. In this particular case, the relations (\ref{R1}) and (\ref{R2}) that determine such extensions take the following form:
\begin{align*}
x_2x_1 & = p_{12}x_1x_2 +p_{11}x_1^2+ atx_1+btx_2+ct^2\\
x_1t   & = dtx_1 + etx_2+ ft^2\\
x_2t    &= gtx_1+htx_2+it^2,  
 \end{align*}
 for certain $a$, $b$, $c$, $d$, $e$, $f$, $g$, $h$, $i \in \K$. Some particular cases of these double Ore extensions $B$ are presented below (see \cite[Example 4.1]{Zhang}):
\begin{itemize}
\item $B^1:=B^1(p,a,b,c)$ where $a$, $b$, $c$, $p \in \K$ and $p\neq 0$, $b\neq 0,1$. The relations for this $B^1$ are
\begin{align*}
x_2x_1 & = px_1x_2 +\dfrac{bc}{1-b}(pb-1)tx_1^2+ at^2,\\
x_1t   & = btx_1, \\
x_2t    &= b^{-1}tx_2+ct^2.  
 \end{align*}
The morphism $\sigma$ 
is given by  $\sigma(t) =  \left(\begin{smallmatrix}
bt && 0 \\
0 && b^{-1}t \\
\end{smallmatrix}\right)$. In this case $\sigma_{12}=\sigma_{21}=0$ and $\sigma_{11}$, $\sigma_{22}$ are automorphisms. In addition, the derivation $\delta$ is determined by $\delta(t) = \left(\begin{smallmatrix}
0 \\
ct^2 \\
\end{smallmatrix}\right)$. The parameter $P$ is $(p,0)$ and the tail is $\tau = \left\{\dfrac{bc}{1-b}(pb-1)t, 0,at^2\right\}$. 
 \item $B^2:=B^2(a,b,c)$, where $a$, $b$, $c\in \K$ and $b\neq 0$. In this instance
\begin{align*}
x_2x_1 & = -x_1x_2 + at^2,\\
x_1t   & = b^{-1}tx_2 +ct^2, \\
x_2t    &= btx_1+(-bc)t^2.
\end{align*}
The morphism $\sigma$ is determined by $\sigma(t) = \left(\begin{smallmatrix}
0 & b^{-1}t \\
bt & 0 \\
\end{smallmatrix}\right)$, hence $\sigma(t^n) = \left(\begin{smallmatrix}
0 & b^{-1}t^n \\
bt^n & 0 \\
\end{smallmatrix}\right)$, when $n$ is odd, and $\sigma(t^n) = \left(\begin{smallmatrix}
t^n & 0 \\
0 & t^n \\
\end{smallmatrix}\right)$ when $n$ is even. Unlike $B^1$, neither $\sigma_{12}$ nor 
$\sigma_{21}$ are  algebra morphisms. Furthermore, the derivation $\delta$ is determined by $\delta(t) = \left(\begin{smallmatrix}
ct^2 \\
-bct^2 \\
\end{smallmatrix}\right)$, the parameter $P$
 is $(-1,0)$, and the tail $\tau =\{0,0,at^2\}$.  It can be proven that $B^2(a,b,c)$ is not an iterated Ore extension of $A=\K[t]$ for a wide range of values for $(a,b,c)$, see \cite[pg. 2685]{Zhang}.
\end{itemize}
Other particular cases of such double Ore extensions of $\K[t]$ are considered in \cite{Zhang} --being these denoted by $B^3$ and $B^4$. Unlike $B^2$, the algebras $B^1$, $B^3$ and $B^4$ turn out to be iterated extensions of $A$.
\end{ejemplo}

\subsection*{The PI property in the double Ore extension B(h)}
\noindent
In \cite{Zhang}, the authors introduced the algebra $B(h)$, described below, and pointed out that this algebra is a novel example satisfying several homological properties of relevance, such as being Artin-Schelter regular, Koszul, Auslander regular, and a Cohen-Macaulay domain. Furthermore, they highlight that when $h$ is a root of unity, this algebra is a PI algebra.
\begin{ejemplo}[{\cite[Example 4.2]{Zhang}}] 
Let $h$ be a nonzero scalar in $\K$ and define
$B(h)$ as the graded $\K$-algebra generated by $x_1$, $x_2$, $y_1$ and $y_2$ subject to the following relations:
\begin{alignat*}{4}
x_2 x_1 &= {-x_1 x_2} & & & & \\
y_2 y_1 &= {-y_1 y_2} & & & & \\
y_1 x_1 &= h( & x_1 y_1 &+ x_2 y_1 &+ x_1 y_2 & \;) \\
y_1 x_2 &= h( & & & + x_1 y_2 & \;) \\
y_2 x_1 &= h( & & + x_2 y_1 & & \;) \\
y_2 x_2 &= h( & & - x_2 y_1 &- x_1 y_2 + x_2 y_2 & \;).
\end{alignat*}
\end{ejemplo}
By \cite[Proposition 4.6]{Zhang}, this $B(h)$ is a right double Ore extension of the following algebra
\[
A := \K_{-1}[x_1,x_2] =  \K\langle x_1, x_2\rangle \big/ \langle x_1x_2 + x_2x_1\rangle.\]

The following result was proved by Zhang in \cite[Corollary~4.8]{Zhang}. We include its proof for completeness, making explicit the finite-over-center argument used to establish the PI property.

\begin{teorema}[{\cite[Corollary 4.8]{Zhang}}]\label{coro4.8}
Let $B=B(h)$. Then $B$ is a PI algebra if and only if $h$ is a root of unity.
\end{teorema}
\begin{proof}
Suppose that $B$ is a PI algebra, and consider the subalgebra $\widetilde{B}$ of $B$ generated by $y_1$ and $u:=x_1 x_2$; thus $ \widetilde{B}$ is also a PI algebra. From the defining relations of $B$, it follows that
\begin{align*}
y_1\,u=(-h^2)\,u\,y_1.
\end{align*}

Observe that $\widetilde{B}$ is an Ore extension of $\K[u]$. Specifically,   $\widetilde{B}= \K[u][\,y_1;\sigma\,]$, with $\sigma(u)=(-h^2)u$ and $\delta= 0$. Proposition 2.5 in \cite{LeroyMatczuk2007} ensures that $\widetilde{B}$ is PI if and only if
$\sigma|_{Z(\K[u])}$ has finite order.  Since $Z(\K[u])=\K[u]$, this condition is equivalent to the existence of some $n\geq 1$ such that $\sigma^n=\mathrm{id}_{\K[u]}$. Applying induction on $n$, one obtains $\sigma^n(u)=\lambda^n u$, where $\lambda = -h^2$, so that $\sigma^n(u)=\lambda^n u=u$ holds if and only if $\lambda^n=1$; that is,  when
$\lambda=-h^2$ is a root of unity. This implies that $(-h^{2})^{n}=1$ for some $n\ge 1$; i.e., $h$ is a root of unity.\\
Conversely, suppose that $h^\ell=1$, for some $\ell\ge 1$.
Let
\[
u:=x_1x_2,\qquad s:=x_1^2+x_2^2,\qquad
v:=y_1y_2,\qquad t:=y_1^2+y_2^2.
\]
From the defining relations of $B$ the following holds:
\begin{align}\label{commutausvt}
\begin{split}
y_i\,u&=(-h^2)\,u\,y_i, \\
y_i\,s&=h^2\,s\,y_i,\\
x_j\,v&=(-h^{-2})\,v\,x_j,\\
x_j\,t&=(h^{-2})\,t\,x_j,
\end{split}
\end{align}
for $i,j\in \{1,2\}$.

Observe that $(x_1\pm x_2)^2=x_1^2+x_2^2=s$ since the cross terms vanish upon applying the relation $x_2x_1=-x_1x_2$. In addition, since $x_2x_1=-\,x_1x_2$, the element $u$ anticommutes with $x_1$ and $x_2$, so that  $u^2$ commutes with $x_1$ and $x_2$. Moreover, the element  $s$ also commutes with $x_1$ and $x_2$. Similarly,  since $y_2y_1=-\,y_1y_2$, the element $v$ anticommutes with $y_1$ and $y_2$,  from which it follows that $v^{2}$  commutes with  $y_1$ and  $y_2$. Finally, the element $t$ also commutes with $y_1$ and $y_2$. 

From (\ref{commutausvt}), and by induction on $n\in \Z^+$, the following relations are obtained:
  \[
    y_i\,u^{n}=(-h^{2})^{n}\,u^{n}y_i,
    \qquad
    y_i\,s^{n}=(h^{2})^{n}\,s^{n}y_i.
  \]
In particular, for $n=2\ell$, and using the fact that $h^{\ell}=1$,  it follows that
  \[
    y_i\,u^{2\ell}=u^{2\ell}y_i,
    \qquad
    y_i\,s^{\ell}=s^{\ell}y_i.
  \] 
Furthermore, since $u^{2}$ and $s$ commute with
 $x_1,x_2$, then $u^{2\ell}$ and $s^{\ell}$ also  commute with $x_1,x_2$. This shows that the elements   $u^{2\ell}$ and $s^{\ell}$  commute with all generators of $B$; that is, these elements belong to the center of $B$.  An analogous argument shows that $v^{2\ell}$ and $t^{\ell}$  also belong to the center of $B$.
 
Now, let
\[
U:=u^{2\ell},\qquad S:=s^\ell,\qquad V:=v^{2\ell},\qquad T:=t^\ell;
\]
then  $B$ is finitely generated as a right module over the subalgebra $\K[U,S,V,T]\subseteq Z(B)$.  

Indeed, if $I$ denotes the bilateral ideal of $B$ generated by $U$, $S$, $T$ and $V$, Observe that every monomial $x_1^{a_1} x_2^{a_2}$ is a $\K$-linear combination of elements of the form:
\[
u^{2i} s^j \quad \text{and} \quad u^{2i} s^j \, x_1 \quad (\text{or } x_2),
\]
where $i,j\geq 0$. Since $U=u^{2\ell}$  and $S=s^\ell$  belong to $I$, module $I$, the exponents $i$ and $j$ are bounded by $0\leq i, j<\ell$, leaving only finitely many monomials. It can also be proven that   every monomial $y_1^{b_1} y_2^{b_2}$ is a $\K$-linear combination of elements of the form:
\[
v^{2i} t^j \quad \text{and} \quad v^{2i} t^j \, y_1 \quad (\text{or } y_2),
\]
where $i,j\geq 0$. Thus, module $I$, every monomial in $y_1$, $y_2$ is a $\K$-linear combination of:
\[
v^{2i} t^j \quad \text{and} \quad v^{2i} t^j \, y_1 \quad (\text{or } y_2),
\]
with $0\leq i,j <\ell$.
Therefore, the following finite set
\[
\mathcal{B} :=
\left\{
(u^2)^r s^a x_1^{\varepsilon_1}x_2^{\varepsilon_2}
(v^2)^c t^d y_1^{\eta_1}y_2^{\eta_2}
\mid
0\leq r,a,c,d<\ell,\ \varepsilon_i,\eta_i\in\{0,1\}
\right\}
\] 
is a spanning set for the quotient algebra $B/I$. Hence $B/I$ is
finite-dimensional. \\
\\
Then, for each polynomial \(b \in B(h)\),
\[
b=\sum_{\alpha,\beta,\gamma,\delta} w_{\alpha\beta\gamma\delta}\,U^\alpha S^\beta V^\gamma T^\delta,
\;\;\; w_{\alpha\beta\gamma\delta} \in \K\!\left\langle \mathcal{B}\right\rangle
.
\] Under these conditions \cite[\S 13.1.13, Corollary]{mcconnell} implies that  $B(h)$ is a PI algebra.
\end{proof}


\section{Two-parameter quantum Heisenberg algebra} \label{Heisenbergdedosparametros}
The three-dimensional Heisenberg Lie algebra $\mathfrak{h}$ over $\K$ with basis $\{x,y,t\}$
subject to the relations 
$[x,y]=t$, $[x,t]=0$, and $[y,t]=0$ is a nilpotent Lie algebra that arises naturally in quantum mechanics in the context of the harmonic oscillator problem. More precisely, its enveloping algebra $U(\mathfrak{h})$ --the associative $\K$-algebra generated by $x$, $y$, $t$, and relations $xy - yx = t$, $xt = tx$, and  $yt = ty$-- admits a central element $t$ such that  $U(\mathfrak{h})/\langle t-1\rangle$ is isomorphic to the first Weyl algebra $A_1(\K)=\K\langle x,y \mid xy-yx = 1\rangle$. The so-called oscillator representation of this  algebra provides the standard algebraic solution to the quantum harmonic oscillator (e.g., \cite[Section 1]{ks}).
\\
The study of quantum analogs of $U(\mathfrak{h})$ has been an active area of research: a one-parameter deformation, denoted by $H_q$, was introduced and studied by Kirkman and Small \cite{ks}. For $q\in\K^*$, this $\K$-algebra $H_q$ is generated by $x$, $y$, and $t$, subject to the relations
\begin{align} \label{quantumonepar}
tx = q^{-1}xt,\qquad ty = qyt,\qquad yx - qxy = t.
\end{align}
They also proved that 
the quotient ring $H_q/\langle \Omega-1 \rangle$, for an appropriate central element $\Omega$, is isomorphic to the Hayashi 
$q$-analog of the Weyl algebra; see \cite{hayashi}.

The algebra $H_q$ is an iterated Ore extension of the form $\K[t][x;\sigma_x][y;\sigma_y,\delta_y]$, where \(\sigma_x(t)=qt\), \(\sigma_y(t)=q^{-1}t\), \(\sigma_y(x)=qx\),
\(\delta_y(t)=0\) and \(\delta_y(x)=t\). Proposition 1.2 in \cite{ks} claims that $H_q$ is a graded Noetherian domain of global dimension  three, as well as an Artin--Shelter regular algebra. Some of these results were subsequently extended to the two-parameter case by Gaddis \cite{Gaddis2016}, who introduced the algebra $H_{p,q}$, where $p,q\in \K^*$. This algebra ${H}_{p,q}$ is generated by $x$, $y$ and $t$, satisfying the relations below
\begin{align}\label{quantumtwopar}
tx = p^{-1}xt,\qquad ty = pyt,\qquad yx - qxy = t.
\end{align} 
The algebra $H_{p,q}$ can be presented as the iterated Ore extension
$$\K[t][x;\sigma_x][y;\sigma_y,\delta_y],
$$
where \(\sigma_x(t)=p\,t\), \(\sigma_y(t)=p^{-1}t\), \(\sigma_y(x)=q\,x\),
\(\delta_y(t)=0\) and \(\delta_y(x)=t\); see \cite[Proposition 3.4]{Gaddis2016}. Therefore ${H}_{p,q}$  is a prime Noetherian domain; moreover, by assigning degree one to the variables $x$ and $y$, and degree two to the variable $t$, then $H_{p,q}$ 
becomes a connected graded algebra.

\begin{observacion}
Let $\K$ be a field and $p$, $q\in \K^*$ with $q\neq p^{-1}$. The \textbf{($p$,$q$)-number} is defined as  
\begin{align}
\label{pqdef}
[n]_{p,q} = \frac{q^n - p^{-n}}{q - p^{-1}} 
= \sum_{i=0}^{n-1} q^i p^{-(n-i)}.
\end{align}
In the case where $p=q^{-1}$, expression (\ref{pqdef}) reduces to the following:
\begin{align}
[n]_{p,q}=p^{-n}\sum_{i=0}^{n-1}1=p^{-n}\,n.   
\end{align}
Additionally,
\begin{align}\label{obs2}
[n+1]_{p,q}
&=q^{n}+p^{-1}[n]_{p,q}.
\end{align}
\end{observacion}
These ($p$,$q$)-numbers satisfy the following useful property.
\begin{lema} If $p$ and $q$ are both primitive roots of unity whose orders  
$\operatorname{ord}(p)$ and $\operatorname{ord}(q)$ divide $n$, then $[n]_{p,q}=0$.
\end{lema}
\begin{proof}
Let $r:=\operatorname{ord}(p)$ and $s:=\operatorname{ord}(q)$. Since
$r\mid n$ and $s\mid n$, it follows that $p^n=1$ and $q^n=1$. For
$\tau:=pq$, one obtains $\tau^n=(pq)^n=p^nq^n=1$. Moreover, the condition $q\neq p^{-1}$ implies that $\tau\neq 1$.
Therefore,
\[
\sum_{i=0}^{n-1}\tau^i=\frac{\tau^n-1}{\tau-1}=0.
\]
Thus,
\[
[n]_{p,q}
=
p^{-n}\sum_{i=0}^{n-1}(pq)^i
=
p^{-n}\sum_{i=0}^{n-1}\tau^i
=
p^{-n}\cdot 0
=
0.
\]
\end{proof}

\begin{lema}[{\cite[Lemma 3.7]{Gaddis2016} }]\label{lema3.7}
Given $n\in \Z^+$, the following identities hold in \(\Hpq\):
\begin{align}
yx^{n} &= q^{n}x^{n}y+\pqnum{n}\,x^{n-1}t, \label{eq:38}\\
y^{n}x &= q^{n}xy^{n}+\pqnum{n}\,ty^{\,n-1}. \label{eq:39}
\end{align}
\end{lema}

\begin{proof}
These identities are proved by induction on \(n\).\\ 
If \(n=1\), 
\[
yx=qxy+t=q^{1}x^{1}y+[1]_{p,q}\,x^{0}t,
\]
where $
[1]_{p,q}=\frac{q-p^{-1}}{q-p^{-1}}=1$.\\
Assume that \eqref{eq:38} holds for \(n=k\),
\[
yx^{k}=q^{k}x^{k}y+\pqnum{k}\,x^{k-1}t.
\]
Multiplying on the right by $x$ and applying (\ref{obs2}), one obtains
\begin{align*}
yx^{k+1}
&=\big(q^{k}x^{k}y+\pqnum{k}\,x^{k-1}t\big)x \\
&=q^{k}x^{k}(yx)+\pqnum{k}\,x^{k-1}(tx) \\
&=q^{k}x^{k}(qxy+t)+\pqnum{k}\,x^{k-1}(p^{-1}xt) \\   
&=q^{k+1}x^{k+1}y+q^{k}x^{k}t+p^{-1}\pqnum{k}\,x^{k}t \\
&=q^{k+1}x^{k+1}y+\big(q^{k}+p^{-1}\pqnum{k}\big)x^{k}t \\
&=q^{k+1}x^{k+1}y+\pqnum{k+1}\,x^{k}t.
\end{align*}
\medskip 
An analogous argument establishes the identity in \eqref{eq:39}.
\end{proof}

The PI property for the algebra $H_{p,q}$ is presented below.
\begin{proposicion}[{\cite[Proposition 3.10]{Gaddis2016}}] 
Let $p$, $q \in \K^*$ with $pq \neq 1$. The algebra $H_{p,q}$ is a PI algebra if and only if $p$ and $q$ are roots of unity.
\end{proposicion}
\begin{proof}
Suppose that $H:=H_{p,q}$ is a PI algebra. Since $t$ is a normal element in $H$, it follows that $tH$ is a two-sided ideal and $H/tH$ is a quotient algebra. In particular,  $\overline{y}\, \overline{x}=q\, \overline{x}\,\overline{y}$; that is, $H/tH$ is isomorphic to the quantum plane with parameter $q$,
\[
H/tH\ \cong\ \mathcal O_q(\K^2):=\K\langle x,y\mid yx=q\,xy\rangle.
\] 
Since $H$ is a PI algebra, the quotient $H/tH\ \cong\ \mathcal O_q(\K^2)$ is also a PI algebra\footnote{It can be proved that every homomorphic image of a PI algebra is PI; see  \cite{DrenskyFormanek2004}.}. Nevertheless, by \cite[Section \S 7.1]{DeConciniProcesi1993}, the algebra $\mathcal O_q(\K^2)$ satisfies the PI property if and only if $q$ is a root of unity.

On the other hand, if $\theta:=(1-pq)\,yx-t$, then $\theta$ is a normal element in $H$. More explicitly, the following relations hold:
\begin{align*}
\theta x = q\,x\theta,\qquad  \theta y = q^{-1}\,y\theta, \qquad \theta t = t\theta.  
\end{align*}
The proof of the second identity is presented below. The remaining ones follow by analogous arguments, see (\ref{quantumtwopar}).

\begin{align*}
\theta y
&= \big((1-pq)yx - t\big)y 
 = (1-pq)yxy - ty \\
&= (1-pq)\big(qxy^2 + ty\big) - ty  \\
&= q(1-pq)xy^2 + \big((1-pq)-1\big)ty \\
&= q(1-pq)xy^2 - pq\,ty \\
&= q(1-pq)xy^2 - pq\,ty +ty -pyt \qquad(ty=pyt) \\
&= q(1-pq)xy^2 +(1-pq)\,ty -pyt\\
&= q^{-1}(1-pq)( q^{2}xy^2 + q ty) -pyt \\
&= q^{-1}(1-pq)( q^{2}xy^2 + q ty) +(1-pq-1)q^{-1}yt \\
&= q^{-1}(1-pq)(q^{2}xy^2+qty +yt) - q^{-1}yt \\
&= q^{-1}(1-pq)(y^{2}x) - q^{-1}yt \qquad ( y^2x=q^2xy^2+qty+yt)\\
&= q^{-1}y\big((1-pq)yx-t\big) \\
&= q^{-1} y\theta.
\end{align*}

Consequently, $\theta H$ is a two-sided ideal and $H/\theta H$ is an algebra where $\overline{t}=(1-pq) \overline{y}\,\overline{x}$.  Moreover, from the relation $yx=qxy+t$ over $H$ the following in  $H/\theta H$ is obtained
\begin{align*}
0 
&= \bar y \bar x - q\,\bar x \bar y - \bar t \\
&= \bar y \bar x - q\,\bar x \bar y - (1-pq)\,\bar y \bar x 
\\
&= pq\,\bar y \bar x - q\,\bar x \bar y 
\\
&= pq\left(\bar y \bar x - p^{-1}\bar x \bar y\right). 
\end{align*}
This implies that $\bar y\,\bar x=p^{-1}\bar x\,\bar y$, so that the quotient $H/\theta H$ is isomorphic to the quantum plane with parameter $p^{-1}$,
\[
H/\theta H\ \cong\ \mathcal O_{p^{-1}}(\K^2).\] 
An argument analogous to the one employed above for $\mathcal{O}_{q}(\K^2)$ allows one to conclude that $p$ is also a root of unity.\\
\\
Conversely, suppose that $p$ and $q$ are roots of unity with
$n:=\operatorname{ord}(p)$ and  $m:=\operatorname{ord}(q)$, respectively. Using the relations $tx=p^{-1}xt$ and $ty=pyt$ in $H$, together with the fact that $p^n=1$, an induction argument shows that  
$t^{n}x=xt^{n}$ and $t^{n}y=yt^{n}$; that is, the element $t^n$ commutes with both   $x$ and $y$. Analogously, one obtains
$tx^{mn}=p^{-mn}x^{mn}t$ and $ty^{mn}=p^{mn}y^{mn}t$. Since  $p^{mn}=(p^{n})^{m}=1$, then $x^{mn}t=tx^{mn}$ and $y^{mn}t=ty^{mn}$; that is, the elements $x^{mn}$ and $y^{mn}$ commute with $t$. In addition, for $r\in \Z^+$, the Lemma \ref{lema3.7} insures that
\begin{align*} 
yx^{r}=q^{r}x^{r}y+[r]_{p,q}\,x^{r-1}t  \text{\quad and \quad}
y^{r}x=q^{r}xy^{r}+[r]_{p,q}\,ty^{\,r-1}.
\end{align*} 
Setting $r=mn$ and using the fact that $q^{mn}=(q^{m})^{n}=1$, one obtains 
\[
yx^{mn}=x^{mn}y+[mn]_{p,q}\,x^{mn-1}t,
  \;
  y^{mn}x=xy^{mn}+[mn]_{p,q}\,t\,y^{mn-1}.
\] 
If $q\neq p^{-1}$, the ($p$, $q$)-number  $[mn]_{p,q}$ satisfies
\begin{align*}
[mn]_{p,q}=\frac{q^{\,mn}-p^{-mn}}{\,q-p^{-1}\,}
=\frac{1-1}{\,q-p^{-1}\,}=0.
\end{align*}
Thus, $yx^{mn}=x^{mn}y$ and $y^{mn}x=xy^{mn}$. Hence,  the power $x^{mn}$ commutes with  $y$, while $y^{mn}$ commutes with $x$.

From the above, it follows that $x^{mn}$, $y^{mn}$, and $t^{n}$ are central elements in $H$; that is,
$\K[x^{mn},y^{mn},t^{n}]\subseteq Z(H)$. In particular,   every monomial $x^{a}y^{b}t^{c}$ can be written as
\[
(x^{mn})^{u}(y^{mn})^{v}(t^{n})^{w}\cdot x^{i}y^{j}t^{\ell},
\quad
0\le i,j<mn,\;\;0\le \ell<n.
\]
Therefore, the algebra $H$ is finitely generated as a module over $\K[x^{mn},y^{mn},t^{n}]$ by the finite set $\{\,x^{i}y^{j}t^{\ell}\,\}$. Under these conditions \cite[\S 13.1.13, Corollary]{mcconnell} implies that $H$ is a PI algebra.
\end{proof}

\begin{corolario} 
Let $q\in k^*$ and assume that $q^2\neq 1$. The quantum  Heisenberg enveloping algebra  
$H_q$ in (\ref{quantumonepar}) is a PI algebra if and only if  $q$ is a root   of unity.
\end{corolario}
\begin{proof}
Setting $p=q$  yields $H_{p,q}=H_q$. Thus, the result follows by applying the preceding proposition. 
\end{proof}


\section{Quantum Matrix Algebra}\label{matrices2p}

The two-parameter quantum matrix algebras introduced by Takeuchi in \cite{Takeuchi1990} arise from the observation that the then-known one-parameter quantizations of the general linear group $GL(n)$ — namely, the standard quantization dual to the Drinfeld-Jimbo enveloping algebra $U_q(\mathfrak{gl}(n))$, and the Dipper-Donkin quantization — could be unified and generalized within a single two-parameter framework. This algebra is defined as follows: let $\K$ be a field and $\alpha$, $\beta\in \K^*$ be two quantum parameters. The two-parameter quantum matrix algebra 
$M_n(\alpha,\beta)$ is the associative $\K$-algebra generated by $\{X_{ij}\}_{1\le i,j\le n}$, subject to the relations:
\[
\begin{aligned}
X_{ij}X_{ik} &= \alpha\, X_{ik}X_{ij} && \text{if } k<j,\\
X_{ij}X_{kj} &= \beta\, X_{kj}X_{ij} && \text{if } k<i,\\
X_{ij}X_{st} &= \beta\alpha^{-1}\, X_{st}X_{ij} && \text{if } s<i,\; j<t,\\
X_{ij}X_{st} - X_{st}X_{ij} &= (\beta-\alpha^{-1})\, X_{sj}X_{it} && \text{if } s<i,\; t<j.
\end{aligned}
\] 
It can be shown that the algebra  $M_n(\alpha,\beta)$ is Artin-Schelter regular algebra, with global  dimension  and  Gelfand--Kirillov dimension both equal to $n^2$, see \cite[Theorem 2]{ArtinSchelterTate1991}.

In the remainder of this section, the case $n=2$ will be considered.
\begin{definicion}
Let $\alpha,\beta \in \K^*$. 
The \textbf{two-parameter quantum matrix algebra} $M_2(\alpha,\beta)$ is the associative $\K$-algebra generated by the entries of the matrix
\begin{align}\label{qmatrix}
X=\begin{pmatrix}
X_{11} & X_{12} \\
X_{21} & X_{22}
\end{pmatrix},
\end{align}
subject to the relations:
\begin{align*}
X_{12}X_{11} &= \alpha X_{11}X_{12} 
&\qquad X_{22}X_{21} &= \alpha X_{21}X_{22} \\
X_{21}X_{11} &= \beta X_{11}X_{21} 
&\qquad X_{22}X_{12} &= \beta X_{12}X_{22} \\
X_{21}X_{12} &= \beta \alpha^{-1} X_{12}X_{21} 
&\;\; X_{22}X_{11} - X_{11}X_{22} &= (\beta - \alpha^{-1})X_{12}X_{21}
\end{align*}
\end{definicion}
Note that if $\alpha=\beta=1$, then the entries of the matrix in (\ref{qmatrix}) commute.

Consequently,  $M_2(1,1)\cong \K[X_{11},X_{12},X_{21},X_{22}]$.
The algebra $M_2(\alpha,\beta)$ admits a presentation as an iterated Ore extension, considering the following ordering of the variables 
$X_{11},X_{12},X_{21},X_{22}$, namely
\[
\K[X_{11}]\,[X_{12},\sigma_{12}]\,[X_{21},\sigma_{21}]\,[X_{22},\sigma_{22},\delta_{22}],
\]
where $\sigma_{12}$, $\sigma_{21}$, and $\sigma_{22}$ are $\K$-linear automorphisms and $\delta_{22}$ is a $\sigma_{22}$-derivation. These morphisms satisfy
\[
\begin{aligned}
\sigma_{12}(X_{11})&=\alpha X_{11},        & \sigma_{21}(X_{11})&=\beta X_{11},        & \sigma_{21}(X_{12})&=\beta\alpha^{-1} X_{12},\\
\sigma_{22}(X_{11})&=X_{11},               & \sigma_{22}(X_{12})&=\beta X_{12},        & \sigma_{22}(X_{21})&=\alpha X_{21},\\
\delta_{22}(X_{11})&=(\beta-\alpha^{-1})X_{12}X_{21}, & \delta_{22}(X_{12})&=0, & \delta_{22}(X_{21})&=0.
\end{aligned}
\]
In this way, $M_2(\alpha,\beta)$ is a prime Noetherian domain and the family of monomials $\{\,X_{11}^a X_{12}^b X_{21}^c X_{22}^d \;\mid \; a,b,c,d\ge 0\,\}
$ forms a $\K$-basis for this algebra. Note that the defining relations of the algebra guarantee that $X_{12}$ and $X_{21}$ are normal elements.

\begin{lema}[{\cite[Lemma 2.2]{BeraMukherjee2025}}]\label{lemam2}
The following identities hold in 
$M_2(\alpha,\beta)$: 
\begin{align}
X_{22}^{k}X_{11}&=X_{11}X_{22}^{k}
+\alpha^{-1}\!\left((\alpha\beta)^{k}-1\right)\,X_{12}X_{21}X_{22}^{k-1}, \label{eq:lem22-1}\\[2mm]
X_{22}X_{11}^{k}&=X_{11}^{k}X_{22}
+\beta\!\left(1-(\alpha\beta)^{-k}\right)\,X_{12}X_{21}X_{11}^{k-1}. \label{eq:lem22-2}
\end{align}
for all $k\geq 1$.
\end{lema}
\begin{proof}
The identity (\ref{eq:lem22-1}) will be proved by induction on $k$. The proof of the second identity is analogous. For the base case, 
\[
X_{22}X_{11} = X_{11}X_{22} + (\beta - \alpha^{-1})X_{12}X_{21},
\]
with $\alpha^{-1}\left((\alpha \beta)^{1} - 1\right) 
= \alpha^{-1}(\alpha \beta - 1) 
= \beta - \alpha^{-1}$. 
Thus \eqref{eq:lem22-1} hold for $k=1$. Now, assume that \eqref{eq:lem22-1} holds for $k\ge1$; then,  
\begin{align*}
X_{22}^{k+1}X_{11}
&=X_{22}\big(X_{22}^{k}X_{11}\big) \\
&=X_{22}\!\Big(X_{11}X_{22}^{k} 
+\alpha^{-1}\!\big((\alpha\beta)^{k}-1\big)X_{12}X_{21}X_{22}^{k-1}\Big) \\
&=(X_{22}X_{11})X_{22}^{k} 
+\alpha^{-1}\!\big((\alpha\beta)^{k}-1\big)\,X_{22}(X_{12}X_{21})X_{22}^{k-1} \\
&=\big(X_{11}X_{22}+(\beta-\alpha^{-1})X_{12}X_{21}\big)X_{22}^{k}+\alpha^{-1}\!\big((\alpha\beta)^{k}-1\big)\,(\alpha\beta)\,X_{12}X_{21}X_{22}^{k} \\
&=X_{11}X_{22}^{k+1}+\Big[(\beta-\alpha^{-1})+\alpha^{-1}\!\big((\alpha\beta)^{k}-1\big)(\alpha\beta)\Big]X_{12}X_{21}X_{22}^{k} \\
&=X_{11}X_{22}^{k+1}
+\alpha^{-1}\!\big((\alpha\beta)^{k+1}-1\big)\,X_{12}X_{21}X_{22}^{k}.
\end{align*} 

In the last line, use has been made of the fact that:
\begin{align*}
(\beta-\alpha^{-1})+\alpha^{-1}\left((\alpha\beta)^{k}-1\right)(\alpha\beta)
&= -\alpha^{-1}+\alpha^{-1}(\alpha\beta)^{k+1} = \alpha^{-1}\left((\alpha\beta)^{k+1}-1\right).
\end{align*}
\end{proof}

\begin{proposicion}[{\cite[Corollary 2.3]{BeraMukherjee2025}}]\label{centro} If $\alpha$ and $\beta$ are $l$-th roots of unity, then $X_{11}^l$, $X_{12}^l$, $X_{21}^l$ and $X_{22}^l$  are central elements in $M_2(\alpha,\beta)$.
\end{proposicion}
\begin{proof}
To establish this, it suffices to show that $X_{ij}^{l}$  commutes with each of the generating elements of $M_2(\alpha,\beta)$, for $1\leq i,j\leq 2$. An induction argument on $m$ yields the following identities:
\[
X_{11}X_{12}^m=\alpha^{-m}X_{12}^mX_{11},\qquad
X_{22}X_{12}^m=\beta^{\,m}X_{12}^mX_{22},\]
\[
X_{21}X_{12}^m=(\beta \alpha ^{-1})^{m}X_{12}^mX_{21}.
\] 
Setting $m=l$ and using the hypotheses on $\alpha$ and $\beta$, the fact that 
 $X_{12}^l \in Z(M_2(\alpha,\beta))$ follows.
Similarly, for the element $X_{21}$, one has:
\[
X_{11}X_{21}^m=\beta^{-m}X_{21}^mX_{11},\qquad
X_{22}X_{21}^m=\alpha^{m}X_{21}^mX_{22},\]
\[ X_{12}X_{21}^m=(\alpha \beta^{-1})^{m}X_{21}^mX_{12}.
\] 
Thus, taking  $m=l$ one obtains  $X_{21}^l\in Z(M_2(\alpha,\beta))$.
Finally, consider $X_{11}^l$ and $X_{22}^l$. An induction argument on $m$ yields the following equalities: 
\[
X_{11}^mX_{12}=\alpha^{-m}X_{12}X_{11}^m,\qquad
X_{11}^mX_{21}=\beta^{-m}X_{21}X_{11}^m
\]
\[
X_{22}^mX_{12}=\beta^{m}X_{12}X_{22}^m,\qquad
X_{22}^mX_{21}=\alpha^{m}X_{21}X_{22}^m.
\]
Once again, setting $m=l$ in the above identities, and $k=l$ in those equations from Lemma \ref{lemam2}, it follows that $X_{11}^l$, $X_{22}^l\in Z(M_2(\alpha,\beta))$. 
\end{proof}
Regarding the PI property of these algebras, the following result holds.
\begin{proposicion}
$M_2(\alpha,\beta)$  is a PI algebra if and only if $\alpha$ and $\beta$ are roots of unity.
\end{proposicion}
\begin{proof}
Suppose that $M_2(\alpha,\beta)$ is a PI algebra. Since every subalgebra
of a PI algebra is again PI, the subalgebras considered below are also PI algebras. Let $A_{12}:=\K\langle X_{11},X_{12}\rangle$. Since  $X_{12}X_{11}=\alpha\,X_{11}X_{12}$ one has $A_{12}\cong \mathcal O_q(\K^2) =  \K\langle x,y\mid yx=qxy\rangle$; that is, $A_{12}$ is the quantum plane with $q=\alpha$. Analogously, letting $A_{21}:=\K\langle X_{11},X_{21}\rangle$, the relation $X_{21}X_{11} = \beta X_{11}X_{21}$ implies  $A_{21}\cong \mathcal O_{q'}(\K^2)$ with $q'=\beta$.  Under these conditions, the results in \cite[Section \S 7.1]{DeConciniProcesi1993} ensure that both $\alpha$ and $\beta$ are roots of unity.

Conversely, suppose that $\alpha$ and $\beta$  are roots of unity with orders $m$ and $n$, respectively, and let $L=mn$. Then $\alpha^L=\beta^L=1$. By Proposition \ref{centro}, the elements 
$X_{11}^L$, $X_{12}^L$, $X_{21}^L$, $X_{22}^L$ are central in  $M_2(\alpha,\beta)$ and, therefore,
\[
\K[X_{11}^L,X_{12}^L,X_{21}^L,X_{22}^L]\;\subseteq\; Z\!\left(M_2(\alpha,\beta)\right).
\]
Additionally, note that for each monomial $X_{11}^aX_{12}^bX_{21}^cX_{22}^d$, there exist  $0\le i,j,k,\ell<L$ such that
\[
a=uL+i,\quad b=vL+j,\] \[ c=wL+k,\quad d=tL+\ell,
\]
for certain $u$, $v$, $w$, $t\in \N$.
Using the commutativity of the 
$L$-th powers and the reordering relations of the algebra, one obtains
\[
X_{11}^aX_{12}^bX_{21}^cX_{22}^d
=(X_{11}^L)^{u}(X_{12}^L)^{v}(X_{21}^L)^{w}(X_{22}^L)^{t}\cdot
X_{11}^{i}X_{12}^{j}X_{21}^{k}X_{22}^{\ell}.
\]
That it, $M_2(\alpha,\beta)$ is finitely generated as a module over
$\K[X_{11}^L,X_{12}^L,X_{21}^L,X_{22}^L]$ by the finite set
\[
\mathcal{B}:=\{\,X_{11}^{i}X_{12}^{j}X_{21}^{k}X_{22}^{\ell}\;:\;0\le i,j,k,\ell<L\,\}.
\]
In this way, by 
 \cite[\S 13.1.13, Corollary]{mcconnell} the algebra $M_2(\alpha,\beta)$ satisfies the PI property.
\end{proof}


\section{The algebra \texorpdfstring{$U_{q}^{+}(B_2)$}{Uq+}}\label{álgebraU_q+}
The algebra \(U_q^{+}(B_2)\)  was originally defined within the general framework of quantum enveloping algebras developed by Drinfeld \cite{Drinfeld1986} and Jimbo \cite{Jimbo1985} in the 1980s, as part of the theory of quantum groups arising from the study of quantum integrable systems and the Yang-Baxter equation.
Its significance lies in the fact that it constitutes a fundamental example of a rank-$2$ quantum algebra (with two simple roots), associated to the complex simple Lie algebra $\mathfrak{g}$  of type $B_2$. In particular $U_q^{+}(B_2)$ is the positive part of the quantum group $U_q(\mathfrak{g})$, generated by the Chevalley generators $E_1$, $E_2$ corresponding to the simple roots of the root system of type $B_2$. This algebra has been studied in detail by Launois \cite{Launois2007}, who explicitly described its prime and primitive spectra, computed its automorphism group, and obtained quantum analogues of the Weyl algebra as simple quotients. Additionally, in \cite{AndruskiewitschDumas2008} some results related to Hopf algebra automorphism are also developed. 
Recently, some properties have been studied in the case where the parameter $q$ is a root of unity. In particular, Bera and Mukherjee \cite{BeraMukherjeeUqB2} proved that $U_q^{+}(B_2)$ is a PI algebra when $q$ is a primitive $m$-th root of unity with $m\geq 5$, computed its center, and classified its simple modules up to isomorphism. In this section, we provide a detailed proof of the fact that this algebra is a PI algebra.

Formally, the algebra   \(U_q^{+}(B_2)\) is the \(\K\)-algebra generated by two indeterminates  $e_1$ and $e_2$ subject to the quantum Serre relations,
\begin{align*}
e_1^{2}e_2 - (q^{2}+q^{-2})\,e_1 e_2 e_1 + e_2 e_1^{2} &= 0\\
e_2^{3}e_1 - (q^{2}+1+q^{-2})\,e_2^{2} e_1 e_2 + (q^{2}+1+q^{-2})\,e_2 e_1 e_2^{2}   - e_1 e_2^{3} &= 0
\end{align*}

In \cite[Section 3.1.2]{AndruskiewitschDumas2008},  Andruskiewitsch and Dumas introduced a new set of generators for
$U_q^{+}(B_2)$. For this, they defined
$e_{3}=e_{1}e_{2}-q^{2}e_{2}e_{1}
$ and $
z=e_{2}e_{3}-q^{2}e_{3}e_{2}$. Thus, the set 
\(\{\, z^{i} e_{3}^{j} e_{1}^{k} e_{2}^{\ell} \mid i,j,k,\ell \in \mathbb{Z}^{+}\,\}\) 
is a PBW basis for $(U_q^{+}(B_2)$ and the latter turns out to be  the \(\K\)-algebra generated by
\(e_1,e_2,e_3\) and \(z\) subject to the relations:
\begin{align*}
&e_i z = z e_i, \ \text{ for }\, i=1,2,3,
&e_1 e_3 &= q^{-2} e_3 e_1,\\
&e_2 e_1 = q^{-2} e_1 e_2 - q^{-2} e_3, &e_2 e_3 &= q^{2} e_3 e_2 + z.
\end{align*}
Furthermore, \(U_q^{+}(B_2)\) can be presented as an iterated Ore extension, 
\[
U_q^{+}(B_2) \;=\; \K[z,e_3][\,e_1;\sigma_1\,][\,e_2;\sigma_2,\delta_2\,],
\]
where
\begin{align*}
\sigma_1(z) &= z, 
& \sigma_1(e_3) &= q^{-2}e_3, \\[6pt]
\sigma_2(z) &= z, 
& \sigma_2(e_3) &= q^{2}e_3, 
& \sigma_2(e_1) &= q^{-2}e_1, \\[6pt]
\delta_2(z) &= 0, 
& \delta_2(e_3) &= z, 
& \delta_2(e_1) &= -q^{-2}e_3.
\end{align*}

Setting $p=1$ in (\ref{pqdef}), the following $q$ numbers are obtained:
\[
[k]_{q^2}=\frac{q^{2k}-1}{q^2-1},\qquad [k]_{q^{-4}}=\frac{q^{-4k}-1}{q^{-4}-1}.
\]

The identities below will be useful in the sequel:
\begin{align*}
q^{2k}+[k]_{q^2}
&=[k+1]_{q^2},\\
q^{-4k}+[k]_{q^{-4}}
&=[k+1]_{q^{-4}}.
\end{align*}

\begin{lema}{\cite[Lemma 2.1]{BeraMukherjeeUqB2}} \label{lem:B2}
Assume that $q^4\neq 1$. The identities below hold in  $U_q^+(B_2)$:
\begin{enumerate}[\rm (i)]
\item $e_2 e_3^{\,k}=q^{2k} e_3^{\,k} e_2+ [k]_{q^2}\, z\, e_3^{\,k-1}.$
  \item $e_2^{\,k} e_3=q^{2k} e_3 e_2^{\,k}+[k]_{q^2}\, z\, e_2^{\,k-1}.$
  \item $e_2 e_1^{\,k}=q^{-2k} e_1^{\,k} e_2-q^{-2}\,[k]_{q^{-4}}\, e_3 e_1^{\,k-1}.$
  \item $e_2^{k} e_1= q^{-2k} e_1 e_2^{k}- q^{-2} B_k\, e_3 e_2^{k-1}
- C_k\, z\, e_2^{k-2}$, 
where 
\begin{align*}
B_k=\frac{q^{2k}-q^{-2k}}{q^{2}-q^{-2}},
\;\;\;\;\;
C_k = q^{-2(k-1)}\,\frac{(q^{2k}-1)(q^{2(k-1)}-1)}{(q^{4}-1)(q^{2}-1)}.
\end{align*}
for $k \geq 2$.
\end{enumerate}
\end{lema}
\begin{proof}
The proof of each of these identities will be carried out by induction on $k$. First of all, the relations of $U_q^+(B_2)$ guarantee the identities for 
 $k=1$. Additionally,  
\begin{align*}
e_2 e_3^{k+1}
&=(q^{2k} e_3^{k} e_2 + [k]_{q^2}\, z\, e_3^{k-1})\,e_3= q^{2k} e_3^{k}(q^{2} e_3 e_2+z)+[k]_{q^2} z e_3^{k}\\
&= q^{2(k+1)} e_3^{k+1} e_2 + \big(q^{2k}+[k]_{q^2}\big) z e_3^{k}= q^{2(k+1)} e_3^{k+1} e_2 + [k\!+\!1]_{q^2}\, z e_3^{k},
\end{align*}
\begin{align*}
e_2^{k+1} e_3
&= e_2\big(q^{2k} e_3 e_2^{k} + [k]_{q^2}\, z\, e_2^{k-1}\big)= q^{2k} (q^{2} e_3 e_2 + z) e_2^{k} + [k]_{q^2} z e_2^{k}\\
&= q^{2(k+1)} e_3 e_2^{k+1} + \big(q^{2k}+[k]_{q^2}\big) z e_2^{k}= q^{2(k+1)} e_3 e_2^{k+1} + [k\!+\!1]_{q^2}\, z e_2^{k},
\end{align*}
and 
\begin{align*}
e_2 e_1^{k+1}
&= \big(q^{-2k} e_1^{k} e_2 - q^{-2}[k]_{q^{-4}}\, e_3 e_1^{k-1}\big) e_1= q^{-2k}e_1^{k} e_2e_1 - q^{-2}[k]_{q^{-4}}\, e_3 e_1^{k}\\
&= q^{-2k}\big(q^{-2} e_1^{k+1} e_2 - q^{-2} e_1^{k} e_3\big) - q^{-2}[k]_{q^{-4}}\, e_3 e_1^{k}\\
&= q^{-2(k+1)} e_1^{k+1} e_2 - q^{-2}\big(q^{-2k}+[k]_{q^{-4}}\big) e_3 e_1^{k} \\
&= q^{-2(k+1)} e_1^{k+1} e_2 - q^{-2}[k\!+\!1]_{q^{-4}}\, e_3 e_1^{k},
\end{align*}
thus proving identities (i)--(iii).
Now, for $k\ge 2$,
 \begin{align*}
q^{2}B_k+q^{-2k}
&=q^{2}\frac{q^{2k}-q^{-2k}}{q^{2}-q^{-2}}+q^{-2k}=\frac{q^{2k+2}-q^{2-2k}}{q^{2}-q^{-2}}+\frac{q^{-2k}(q^{2}-q^{-2})}{q^{2}-q^{-2}}\\
&=\frac{q^{2k+2}-q^{2-2k}+q^{-2k+2}-q^{-2k-2}}{q^{2}-q^{-2}}\\
&=\frac{q^{2k+2}-q^{-2k-2}}{q^{2}-q^{-2}}
= B_{k+1},
\end{align*}

whereas,
\begin{align*}
q^{-2}B_k
&= q^{-2}\,\frac{q^{2k}-q^{-2k}}{q^{2}-q^{-2}}
= \frac{q^{2k}-q^{-2k}}{q^2\,(q^{2}-q^{-2})}= \frac{q^{2k}-q^{-2k}}{q^{4}-1}
= \frac{(q^{2k}-1)(1+q^{-2k})}{q^{4}-1}.
\end{align*}
Thus, 
\begin{align*}
q^{-2}B_k+C_k
&=\frac{(q^{2k}-1)\big(q^{2}-q^{-2k}\big)}
{(q^{4}-1)(q^{2}-1)} =\frac{(q^{2k}-1)\,q^{-2k}\,(q^{2k+2}-1)}
{(q^{4}-1)(q^{2}-1)}\\[2pt]
&=q^{-2k}\,\frac{(q^{2k+2}-1)(q^{2k}-1)}{(q^{4}-1)(q^{2}-1)}
= C_{k+1}.
\end{align*}
 
Hence, 
\begin{align*}
 q^{-2}B_k + C_k = C_{k+1}.   
\end{align*}
Using the latter equations, it follows that
 \begin{align*}
e_2^{k+1}e_1
&= e_2\left( e_2^{k}e_1 \right)= e_2\!\left( q^{-2k}e_1e_2^{k} - q^{-2}B_k\,e_3e_2^{k-1} - C_k\,z\,e_2^{k-2} \right) \\ 
&= q^{-2k} e_2(e_1e_2^{k})
      - q^{-2}B_k e_2(e_3e_2^{k-1})
      - C_k\,e_2(ze_2^{k-2}) \\
&= q^{-2k}\!\left( (q^{-2}e_1e_2 - q^{-2}e_3) e_2^{k} \right)- q^{-2}B_k\!\left( (\,q^{2}e_3e_2 + z\,) e_2^{k-1} \right)
   \;-\; C_k\, z\, e_2^{k-1} \\
&= q^{-2(k+1)} e_1 e_2^{k+1}
   - q^{-2(k+1)} e_3 e_2^{k}
   - q^{-2}B_k\, q^{2} e_3 e_2^{k}- q^{-2}B_k\, z\, e_2^{k-1}
   \;-\; C_k\, z\, e_2^{k-1} \\
&= q^{-2(k+1)} e_1 e_2^{k+1}
   - q^{-2}\left( q^{-2k} + q^2B_k \right) e_3 e_2^{k}- \left( q^{-2}B_k + C_k \right) z\, e_2^{k-1}\\
&= q^{-2(k+1)} e_1 e_2^{k+1} - q^{-2}\left( B_{k+1} \right) e_3 e_2^{k}
   - \left( C_{k+1} \right) z\, e_2^{k-1}.
\end{align*}
 
\end{proof}

\begin{proposicion}{\cite[Corollary 2.2]{BeraMukherjeeUqB2}}\label{centro2} 
If $q$ is a primitive 
$l$-th root of unity, with  $l \geq 5$, then $z$, $e_{1}^l$, $e_{2}^l$ and  $e_{3}^l$ are central elements of \(U_q^{+}(B_2)\).
\end{proposicion}
\begin{proof}
Since $q$ is a primitive $l$-th root of unity, it follows that
\begin{align*}
 q^{2l}=1,\quad [l]_{q^{2}}=0,\quad [l]_{q^{4}}=0,\quad B_l=0,\quad C_l=0.
\end{align*}
From the defining relations of \(U_q^{+}(B_2)\), it follows that $z$ is a central element. On the other hand, by  Lemma \ref{lem:B2}--(i) with $k=l$ one obtains 
\[
e_2 e_3^{\,l}=q^{2l} e_3^{\,l} e_2 + [l]_{q^{2}} z e_3^{\,l-1}=e_3^{\,l} e_2.
\]
Moreover, from the relation $e_1 e_3=q^{-2} e_3 e_1$ it follows by induction that 
$e_1 e_3^{\,l}=q^{-2l} e_3^{\,l} e_1=e_3^{\,l} e_1$. Hence,  $e_3^{\,l}$ commute with $e_1$ and $e_2$.
Analogously, Lemma \ref{lem:B2}--(ii), for $k=l$, it follows
\[
e_2^{\,l} e_3 = q^{2l} e_3 e_2^{\,l} + [l]_{q^{2}} z e_2^{\,l-1}= e_3 e_2^{\,l}.
\]
Once again, Lemma \ref{lem:B2}--(iv) implies
\[
e_2^{\,l} e_1 = q^{-2l} e_1 e_2^{\,l} - q^{-2} B_l\, e_3 e_2^{\,l-1} - C_l\, z e_2^{\,l-2}
= e_1 e_2^{\,l};
\]
that is, $e_2^{\,l}$ commute with $e_1$ and $e_3$. From the relation $e_1 e_3=q^{-2} e_3 e_1$ can be proved that $e_3 e_1^{\,l}=q^{2l} e_1^{\,l} e_3 = e_1^{\,l} e_3$. Indeed, Lemma \ref{lem:B2}--(iii), assuming $k=l$, implies
\[
e_2 e_1^{\,l}= q^{-2l} e_1^{\,l} e_2 - q^{-2}[l]_{q^{4}} e_3 e_1^{\,l-1}
= e_1^{\,l} e_2.
\]
Thus, $e_1^{\,l}$ commute  with $e_2$ and $e_3$. 
Finally, since  $e_1^{\,l}$, $e_2^{\,l}$, $e_3^{\,l}$ and $z$ commute with the generators of $U_q^{+}(B_2)$, it is concluded that these elements belong to its center.
\end{proof}
In the following proposition, the root-of-unity case is considered under the standing assumption that $q$ is a primitive $\ell$-th root of unity with $\ell\geq 5$, as in Proposition \ref{centro2}.
\begin{proposicion}{\cite[Proposition 2.3]{BeraMukherjeeUqB2}}
 $U_q^{+}(B_2)$ is a PI algebra if and only if $q$ is an $l$-th  root of unity.
\end{proposicion}
\begin{proof}
Consider the subalgebra $A:= \K\langle e_1,e_3\rangle$ of $U_q^{+}(B_2)$. The relation  $e_1e_3=q^{-2}e_3e_1$ shows that $A$ is isomorphic to the quantum plane with parameter $\lambda=q^{-2}$:
\[
A\ \cong\  \K\langle x,y\mid yx=\lambda xy\rangle,\qquad \lambda=q^{-2}.
\]
Hence, by \cite[Section \S 7.1]{DeConciniProcesi1993} one obtains  $q^{-2}$ is a root of unity, which implies that $q$ is a root of unity. Conversely, Proposition \ref{centro2}, ensures that 
\[
R:= \K\!\left[e_1^{\,l},\,e_2^{\,l},\,e_3^{\,l},\,z\right]\subseteq Z \big(U_q^{+}(B_2)\big).
\] 
For a monomial of the form $e_1^{a}e_2^{b}e_3^{c}$, there exist $u,v,w \in \mathbb{N}$ such that 
\[  
a=ul+i,\qquad b=vl+j,\qquad c=wl+k,
\]
with $0\leq i,j,k < l$. Using the identities of Lemma \ref{lem:B2} and the relation $e_1e_3=q^{-2}e_3e_1$, any product 
$e_1^{a}e_2^{b}e_3^{c}$ may be reordered by successively displacing $e_1^{L}$, $e_2^{L}$ and $e_3^{L}$ to the left.
Each time a commutation is performed, the resulting correction terms are either multiplied by $z$
or strictly decrease some of the exponents of $e_1,e_2,e_3$; that is,
 \[
e_1^{a}e_2^{b}e_3^{c}\ \in\
\sum_{0\le i,j,k<l} R\cdot e_1^{i}e_2^{j}e_3^{k}.
\] 
This establishes that $U_q^{+}(B_2)$ is  generated as an $R$-module by the finite set 
\[
\mathcal{B}:=\big\{\,e_1^{i}e_2^{j}e_3^{k}\ \mid \ 0\le i,j,k<l\,\big\}.
\] 
Therefore, 
 \cite[\S 13.1.13, Corollary]{mcconnell}
guarantee that 
$U_q^{+}(B_2)$ sa\-tis\-fies the PI property.
\end{proof}


\section{Quantized Weyl Algebras}\label{Wely cuanticas}
Let $\K$ be a field, $n$  a positive integer, and $\Lambda = (\lambda_{ij})_{n \times n}$  a multiplicatively antisymmetric matrix over $\K$; that is,  $\lambda_{ii} = 1$ and $\lambda_{ij}\lambda_{ji} = 1$ for all $1 \leq i,j \leq n$. 
In addition, let $\mathbf{q} = (q_1, \dots, q_n) \in \K^n$ be  an $n$-tuple with each 
$q_i \in \K \setminus \{0,1\}$. So, given $\Lambda$ and $\mathbf{q}$, the \textbf{multiparametric quantum Weyl algebra} 
$A_{n}^{\mathbf{q},\Lambda}$, introduced by Maltsiniotis in \cite{Maltsiniotis1990} in his work on non-commutative differential calculus, is defined as the $\K$-algebra generated by the variables $x_1, y_1, \dots, x_n, y_n$, subject to the following relations:
\[
\begin{aligned}
x_i x_j &= q_i \lambda_{ij} x_j x_i, \quad && 1 \le i < j \le n,\\
x_i y_j &= \lambda_{ji}^{-1} y_j x_i, \quad && 1 \le i < j \le n,\\
y_i y_j &= \lambda_{ij} y_j y_i, \quad && 1 \le i < j \le n,\\
y_i x_j &= q_i^{-1} \lambda_{ij}^{-1} x_j y_i, \quad && 1 \le i < j \le n,\\
x_i y_i - q_i y_i x_i &= 1 + \sum_{k=1}^{i-1} (q_k - 1) y_k x_k, \quad && 1 \le i \le n.
\end{aligned}
\] 
Another family of multiparametric quantum Weyl algebras has been studied in the literature, see \cite{AkhavizadeganJordan1996}.  This family exhibits more symmetric relations than those of $A_{n}^{\mathbf{q},\Lambda}$. The symmetric multiparametric quantum Weyl algebra $\mathcal{A}_{n}^{\mathbf{q},\Lambda}$, introduced by Akhavizadegan and Jordan and also known as the \textbf{alternative quantum Weyl algebra}, is the $\K$-algebra generated by the variables $x_1,y_1,\ldots,x_n,y_n$ subject to the following relations:
\[
\begin{aligned}
x_i x_j &= \lambda_{ij} x_j x_i, \quad && 1 \le i < j \le n,\\
x_i y_j &= \lambda_{ij}^{-1} y_j x_i, \quad && 1 \le i < j \le n,\\
y_i y_j &= \lambda_{ij} y_j y_i, \quad && 1 \le i < j \le n,\\
y_i x_j &= \lambda_{ij}^{-1} x_j y_i, \quad && 1 \le i < j \le n,\\
x_i y_i - q_i y_i x_i &= 1, \quad && 1 \le i \le n.
\end{aligned}
\]
Both algebras  $A_{n}^{\mathbf{q},\Lambda}$ and $\mathcal{A}_{n}^{\mathbf{q},\Lambda}$ admit presentations as iterated Ore extensions, from which it follows that they are Noetherian affine domains, see  \cite{AkhavizadeganJordan1996, Jordan1995}. When the parameters are roots of unity, it is assumed that $q_i$ is a primitive root of unity of order $l_i$, for all $1 \le i \le n$, and that $\lambda_{ij}$ is a root of unity of order $l_{ij}$, for all $1 \le i < j \le n$.

If $z_0:=1$ and $z_i:=x_i y_i-y_i x_i$ are defined for $1\le i\le n$, then 
\[
z_i \;=\; 1+\sum_{j=1}^i (q_j-1)\,y_jx_j
\;=\; z_{i-1}+(q_i-1)\,y_i x_i.
\]
Indeed, note that
\begin{align*}
z_i
&=x_i y_i-y_i x_i = (x_i y_i-q_i y_i x_i)+(q_i-1)y_i x_i\\
&=\Big(1+\sum_{j=1}^{i-1}(q_j-1)\,y_j x_j\Big)+(q_i-1)y_i x_i\\
&=1+\sum_{j=1}^{i}(q_j-1)\,y_j x_j,
\end{align*}
whereas, 
\begin{align*}
z_i
&= x_i y_i - y_i x_i =  1 + \sum_{j=1}^{i} (q_j-1)\,y_j x_j \\
&= 1 + \sum_{j=1}^{i-1} (q_j-1)\,y_j x_j + (q_i-1)\,y_i x_i \\
&= z_{i-1} + (q_i-1)\,y_i x_i.
\end{align*}
Thus, the last of the defining relations of $A_{n}^{\mathbf{q},\Lambda}$ may be rewritten as follows,
\begin{align*}
x_i y_i - q_i\,y_i x_i 
&= 1 + \sum_{j=1}^{i-1} (q_j-1)\,y_j x_j = z_{i-1}.
\end{align*}

\begin{lema}\label{lemazjcentral} For each  \(1\le i\le n\), the element \(z_i\) is  normal in  $A_{n}^{\mathbf{q},\Lambda}$.

\end{lema}
\begin{proof}
In fact, for  \(i,j\), with \(1\le i<j\le n\), 
\begin{align*}
z_i x_j
&=(x_i y_i - y_i x_i)\,x_j = x_i(y_i x_j) - y_i(x_i x_j) \\
&= x_i\!\big(q_i^{-1}\lambda_{ij}^{-1} x_j y_i\big)
   - y_i\!\big(q_i\lambda_{ij} x_j x_i\big) \\
&= q_i^{-1}\lambda_{ij}^{-1}(x_i x_j) y_i
   - q_i\lambda_{ij}(y_i x_j) x_i \\
&= q_i^{-1}\lambda_{ij}^{-1}\big(q_i\lambda_{ij} x_j x_i\big) y_i
   - q_i\lambda_{ij}\big(q_i^{-1}\lambda_{ij}^{-1} x_j y_i\big) x_i \\
&= x_j x_i y_i - x_j y_i x_i
= x_j(x_i y_i - y_i x_i)
= x_j z_i.
\end{align*}
Likewise,
\[
\begin{aligned}
z_i y_j
&=(x_i y_i - y_i x_i)\,y_j = x_i(y_i y_j) - y_i(x_i y_j) \\
&= x_i(\lambda_{ij} y_j y_i) - y_i(\lambda_{ij}^{-1} y_j x_i) \\
&= \lambda_{ij}(x_i y_j) y_i -\lambda_{ij}^{-1} (y_i y_j) x_i  \\
&=\lambda_{ij}\lambda_{ij}^{-1} y_j x_i y_i - \lambda_{ij}^{-1} \lambda_{ij} y_j y_i x_i  \\
&= y_j(x_i y_i - y_i x_i)
= y_j z_i .
\end{aligned}
\]

Additionally, given \(i,j\),  with \(1\le j < i\le n\), one obtains
\begin{align*}
z_i x_j
&=(x_i y_i-y_i x_i)x_j \\
&=x_i(y_i x_j)-y_i(x_i x_j) \\
&=x_i(\lambda_{ij}\,x_j y_i)-y_i(q_j^{-1}\lambda_{ij}^{-1})\,x_j x_i \\
&= (\lambda_{ij})(x_i x_j)y_i-(q_j^{-1}\lambda_{ij}^{-1})(y_i x_j)x_i \\
&=(\lambda_{ij})(q_j^{-1}\lambda_{ij}^{-1})\,x_j x_iy_i-(q_j^{-1}\lambda_{ij}^{-1})(\lambda_{ij})\,x_j y_ix_i \\
&=\,q_j^{-1}x_j(x_i y_i-y_i x_i)
=\,q_j^{-1}x_j z_i.
\end{align*}
Analogously, 
\begin{align*}
z_i y_j
&=(x_i y_i-y_i x_i)y_j \\
&=x_i(y_i y_j)-y_i(x_i y_j) \\
&=\lambda_{ij}^{-1}(x_iy_j) y_i-q_j\lambda_{ij}(y_i y_j )x_i \\
&= \lambda_{ij}^{-1}q_j\lambda_{ij}y_jx_iy_i-q_j\lambda_{ij}\lambda_{ij}^{-1}y_jy_ix_i  \\
&= q_jy_j(x_iy_i-y_ix_i)
=\,q_jy_j z_i.
\end{align*}
Furthermore, if \(i<j\), 
\begin{align*}
z_i z_j
&= (x_i y_i - y_i x_i) z_j =  x_i (y_i z_j) - y_i (x_i z_j)\\
&= x_i (z_j y_i) - y_i (z_j x_i) = z_j(x_i y_i) - z_j (y_i  x_i)\\
&= z_j(x_i y_i - y_i x_i)  \\
&= z_j z_i;
\end{align*}
 and, if \(j\leq i\), it follows that
\begin{align*}
z_i z_j
&= z_i(x_j y_j - y_j x_j) = (z_i x_j) y_j - (z_i y_j) x_j \\
&= (q_j^{-1} x_j z_i) y_j - (q_j\, y_j z_i) x_j \\
&= q_j^{-1} x_j (z_i y_j) - q_j\, y_j (z_i x_j) \\
&= q_j^{-1} x_j (q_j\, y_j z_i) - q_j\, y_j (q_j^{-1} x_j z_i) \\
&= x_j y_j z_i - y_j x_j z_i = (x_j y_j - y_j x_j) z_i \\
&= z_j z_i.
\end{align*}
\end{proof}

\begin{proposicion}\label{identities6-2} For each \(i\), \(k\ge1\), the following identities hold
\begin{enumerate}[\rm (i)]
\item $ x_i^{k}y_i \;=\; q_i^{k}y_i x_i^{k}
\;+\; (1+q_i+\cdots+q_i^{k-1})\, z_{i-1}\,x_i^{k-1}$.
\item $x_i y_i^{k} \;=\; q_i^{k}y_i^{k} x_i
\;+\; (1+q_i+\cdots+q_i^{k-1})\, z_{i-1}\,y_i^{\,k-1}$.
\end{enumerate}
\end{proposicion}
\begin{proof}
Identity (i) is proved by induction on \(k\); case (ii) is analogous. Note that the case $k=1$ is reduced exactly to the relation \(x_i y_i-q_i y_i x_i=z_{i-1}\). Suppose the identity holds for $k\geq 1$; that is, 
$$x_i^{k}y_i  = q_i^{k}y_i x_i^{k}
 + S_k\, z_{i-1}x_i^{k-1},$$
where $S_k:=1+q_i+\cdots+q_i^{k-1}$. In these conditions,
\begin{align*}
x_i^{k+1}y_i
&=x_i^{k}(x_i y_i)
 =x_i^{k}(q_i y_i x_i+z_{i-1})\\
&= q_i\,x_i^{k}y_i x_i \;+\; x_i^{k}z_{i-1}.
\end{align*}
Applying the induction hypothesis and using Lemma \ref{lemazjcentral}, it fo\-llo\-ws that 
\begin{align*}
x_i^{k+1}y_i
&= q_i\!\left(q_i^{k}y_i x_i^{k}
 + S_k\, z_{i-1}x_i^{k-1}\right)x_i \;+\; z_{i-1}x_i^{k}\\
&= q_i^{k+1} y_i x_i^{k+1}
   + q_i S_k\, z_{i-1}x_i^{k}
   + z_{i-1}x_i^{k}\\
&= q_i^{k+1} y_i x_i^{k+1}
   + (q_i S_k + 1)\, z_{i-1}x_i^{k}\\
&=  q_i^{k+1} y_i x_i^{k+1}
  + S_{k+1}\, z_{i-1}x_i^{k},
\end{align*}
the last equality being a consequence of the fact below:
$$q_i S_k+1= q_i(1+\cdots+q_i^{k-1})+1
= 1+q_i+\cdots+q_i^{k}\,=\,S_{k+1}.$$
\end{proof}

\begin{proposicion}{\cite[Proposition 2.4]{BeraMukherjee2023}} \label{propiedad PI wel multiparámetro}
Let $n\geq 2$. The quantum Weyl algebra $A_{n}^{\mathbf{q},\Lambda}$ (respectively, $\mathcal{A}_{n}^{\mathbf{q},\Lambda}$) 
is a PI algebra if and only if each $q_i$ and each $\lambda_{ij}$ are roots of unity.
\end{proposicion}
\begin{proof} 
The statement is proved for $A_{n}^{\mathbf{q},\Lambda}$.

Suppose that all $q_i$ and $\lambda_{ij}$ are roots of unity, and define
\[
\ell:= \operatorname{lcm}\left\{\,\operatorname{ord}(q_i),\,\operatorname{ord}(\lambda_{ij}) \,\middle|\, 1 \le i < j \le n \right\}.
\]
From the defining relations of the algebra and by induction, the following identities are obtained:  
\[
x_i^{\ell} x_j = (q_i\lambda_{ij})^{\ell}\, x_j x_i^{\ell},\qquad
x_i^{\ell} y_j = \lambda_{ji}^{-\ell}\, y_j x_i^{\ell}.
\] 
Since $\ell$ is a multiple of the orders of $q_i$ and $\lambda_{ij}$, it follows that 
\[
x_i^{\ell} x_j = x_j x_i^{\ell},\qquad x_i^{\ell} y_j = y_j x_i^{\ell}\quad (i\neq j).
\] 
This implies that $x_i^{\ell}$ commutes with $x_j$ and $y_j$, for $j\neq i$.
In a similar way, the identities below are obtained:
\[
y_i^{\ell} x_j \;=\; (q_j \lambda_{i j})^{-\ell} x_j y_i^{\ell},\qquad
y_i^{\ell} y_j \;=\; \lambda_{ij}^{\ell} y_j y_i^{\ell}.
\] 
Since $(q_j \lambda_{ij})^{\ell}=\lambda_{ij}^{\ell}=1$, it follows that, 
\[
y_i^{\ell} x_j = x_j y_i^{\ell},\qquad y_i^{\ell} y_j = y_j y_i^{\ell}\quad (i\neq j).
\] 
Whence,  $y_i^{\ell}$ commutes with $x_j$ and $y_j$, for $j\neq i$.

For  $i=j$,  the fact that
\begin{align*} 1+q_i+\cdots+q_i^{\,\ell-1} = 0,
\end{align*}
together with Proposition \ref{identities6-2}, ensures that 
\begin{align*}   x_i^{\ell}y_i &= q_i^{\ell}y_i x_i^{\ell}
+ (1+q_i+\cdots+q_i^{\ell-1})\, z_{i-1}\,x_i^{\ell-1}=y_i x_i^{\ell},\\
x_i y_i^{\ell}&= q_i^{k}y_i^{\ell} x_i
+ (1+q_i+\cdots+q_i^{\ell-1})\, z_{i-1}\,y_i^{\,\ell-1} = y_i^{\ell} x_i.
\end{align*}
All of the above ensures that $x_i^{\ell}$  and $y_i^{\ell}$ are in $Z\!\bigl(A_{n}^{\mathbf q,\Lambda}\bigr)$, for all $1\le i\le n$. In consequence
\begin{align*}
A := \K\!\left[x_i^{\,\ell},\,y_i^{\,\ell}\right]\subseteq Z\!\big(A_{n}^{\mathbf q,\Lambda}\bigr).
\end{align*}
Let
\[
C:=\K\!\left[
x_1^\ell,\ldots,x_n^\ell,
y_1^\ell,\ldots,y_n^\ell
\right]
\subseteq Z\!\left(A_n^{\mathbf q,\Lambda}\right).
\]
We claim that $A_n^{\mathbf q,\Lambda}$ is finitely generated as a
module over $C$. Indeed, by the PBW basis of
$A_n^{\mathbf q,\Lambda}$, every element is a finite $\K$-linear
combination of ordered monomials of the form
\[
x_1^{a_1}\cdots x_n^{a_n}y_1^{b_1}\cdots y_n^{b_n},
\qquad a_i,b_i\geq 0.
\]
For each $i$, write
\[
a_i=\ell u_i+r_i,\qquad b_i=\ell v_i+s_i,
\qquad 0\leq r_i,s_i<\ell.
\]
Since $x_i^\ell$ and $y_i^\ell$ are central, each ordered monomial can be
written as
\[
x_1^{a_1}\cdots x_n^{a_n}y_1^{b_1}\cdots y_n^{b_n}
=
\left(\prod_{i=1}^n (x_i^\ell)^{u_i}
\prod_{i=1}^n (y_i^\ell)^{v_i}\right)
x_1^{r_1}\cdots x_n^{r_n}y_1^{s_1}\cdots y_n^{s_n}.
\]
Thus $A_n^{\mathbf q,\Lambda}$ is generated as a $C$-module by the finite
set
\[
\mathcal S:=
\left\{
x_1^{r_1}\cdots x_n^{r_n}
y_1^{s_1}\cdots y_n^{s_n}
\ \middle|\
0\leq r_i,s_i<\ell,\ 1\leq i\leq n
\right\}.
\]
In particular, $\#\mathcal S=\ell^{2n}$, and therefore
$A_n^{\mathbf q,\Lambda}$ is finitely generated as a module over the
central subalgebra $C$. Hence, by
\cite[\S 13.1.13, Corollary]{mcconnell}, the algebra
$A_n^{\mathbf q,\Lambda}$ satisfies the PI property.\\

Conversely, suppose that $A_{n}^{\mathbf q,\Lambda}$  satisfies the PI property. Thus every $\K$-subalgebra  
of $A_{n}^{\mathbf q,\Lambda}$
satisfies the PI property. In particular, the subalgebra $\K\langle y_i,y_j \mid y_i y_j=\lambda_{ij} y_j y_i\rangle$ which is isomorphic to the quantum plane with parameter $\lambda_{ij}$. Once again, the results in \cite[Section \S 7.1]{DeConciniProcesi1993} ensure that $\lambda_{ij}$  is a root of unity.  Similarly, the subalgebra  $\K\langle x_i,x_j \mid x_i x_j=(q_i\lambda_{ij}) x_j x_i\rangle$   is isomorphic to the quantum plane with parameter $q_i\lambda_{ij}$,  so that $q_i\lambda_{ij}$ is also a root of unity.  Since the product of roots of unity is again a root of unity, it follows that  $q_i=(q_i\lambda_{ij})\lambda_{ij}^{-1}$ is likewise a root of unity.\\
The proof for the algebra  $\mathcal{A}_{n}^{\mathbf{q},\Lambda}$ follows by analogous arguments. 
\end{proof}


\section{Bi-quadratic algebras on 3 generators}\label{bicuadráticas}
In 2023, Bavula in \cite{Bavula2023} introduced the class of skew bi-quadratic algebras and, as a special case, the class of bi-quadratic algebras, giving an explicit description of those on three generators that admit a PBW basis. The algebras in this family include, as particular instances, some quantum planes, quantum Weyl algebras, and universal enveloping algebras of three-dimensional Lie algebras (see \cite[Section 1]{Bavula2023}).

For a ring $D$ and an integer number $n\ge 2$, a family $M = (m_{ij})_{i>j}$ of elements $m_{ij}\in D$, where $1 \le j < i \le n$,  is called a \textbf{lower triangular half-matrix} with coefficients in $D$. In addition,  $L_n(D)$ will  denote the set  of all such matrices.

\begin{definicion}{\cite[\S 1, pg 696]{Bavula2023}}\label{def:SBQA}
Let $\sigma = (\sigma_{1},\dots,\sigma_{n})$ be an  $n$-tuple of commuting endomorphisms of $D$, 
$\delta = (\delta_{1},\dots,\delta_{n})$ be an $n$-tuple of $\sigma_i$-derivations of $D$
(that is, $\delta_{i}$ is a $\sigma_{i}$-derivation of $D$, for all $i\in \{1,\dots,n\}$). If $Q = (q_{ij}) \in L_{n}(Z(D))$, $\mathbb{A}:= (a_{ij,k})$, where $a_{ij,k} \in D$,
with $1 \le j < i \le n$ and $k\in\{1,\dots,n\}$, and $\mathbb{B} := (b_{ij}) \in L_{n}(D)$, the \textbf{skew bi-quadratic algebra}
(SBQA)
\[
A:= D[x_{1},\dots,x_{n};\,\sigma,\delta,Q,\mathbb{A}, \mathbb{B}]
\]
is a ring generated by $D$ and elements $x_{1},\dots,x_{n}$ subject to the following relations,
\begin{equation}\label{rel:SBQA1}
x_{i} a = \sigma_{i}(a) x_{i} + \delta_{i}(a),
\quad \text{for } i=1,\dots,n,\ \text{and all } a \in D,
\tag{*}
\end{equation}
\begin{equation}\label{rel:SBQA2}
x_{i} x_{j} - q_{ij} x_{j} x_{i}
= \sum_{k=1}^{n} a_{ij,k}\, x_{k} + b_{ij},
\quad \text{for all } j < i.
\tag{**}
\end{equation}
If $\sigma_{i} = \mathrm{id}_{D}$ and $\delta_{i} = 0$ for $i=1,\dots,n$, the ring $A$ is called the 
\textbf{bi-quadratic algebra} (BQA) and it is denoted by
\[
A:= D[x_{1},\dots,x_{n};\,Q,\mathbb{A},\mathbb{B}].
\]
Finally, the algebra $A$ is said to have a \textbf{PBW basis} if
\begin{align*}
A = \bigoplus_{\alpha \in \mathbb{N}^{n}} D x^{\alpha},
\text{ where }
x^{\alpha}:= x_{1}^{\alpha_{1}} \cdots x_{n}^{\alpha_{n}}.
\end{align*}
\end{definicion}
The set $\mathrm{Mon}(A):=\{x^{\alpha} \mid \alpha \in \mathbb{N}^n\}$ of standard monomials over $A$ is a linear order under the degree-lexicographic ordering. Specifically, in \cite{Bavula2023}, Bavula establishes the order $x_{1}<\cdots<x_{n}$ on the variables. Thus, if $A=D[x_{1},\ldots,x_{n};Q,\mathbb{A},\mathbb{B}]$ is a bi-quadratic algebra with $n\ge 3$, then for each triple $i,j,k\in\{1,\ldots,n\}$ such that $i<j<k$, there are exactly two distinct ways to reduce the product $x_{k}x_{j}x_{i}$ with respect to this ordering:
\begin{align}\label{solapamiento}
x_{k}x_{j}x_{i}
  &= q_{kj}q_{ki}q_{ji}x_{i}x_{j}x_{k}
    + \sum_{|\alpha|\le 2} c_{k,j,i,\alpha} x^{\alpha},\\
x_{k}x_{j}x_{i}
  &= q_{kj}q_{ki}q_{ji}x_{i}x_{j}x_{k}
    + \sum_{|\alpha|\le 2} c'_{k,j,i,\alpha} x^{\alpha}.\notag
\end{align}
Regarding the PBW basis property, it is proved in \cite[Theorem 1.1]{Bavula2023} that the defining relations (\ref{rel:SBQA1}) and  (\ref{rel:SBQA2}) of the algebra $A$ are consistent (that is, the two reduction paths for $x_k x_j x_i$ described above in (\ref{solapamiento}) yield the same result, so that $A$ does not collapse to the trivial algebra) and that $A= \bigoplus_{\alpha\in\mathbb{N}^{n}} D x^{\alpha}$ if and only if, for all triples $i,j,k\in\{1,\ldots,n\}$ such that $i<j<k$, the condition
 $c_{k,j,i,\alpha}=c'_{k,j,i,\alpha}$ is satisfied.

Let $A = \K[x_{1},x_{2},x_{3};Q,\mathbb{A},\mathbb{B}]$ be a bi-quadratic algebra, where $Q=(q_{1},q_{2},q_{3})\in (\K^{*})^{3}$,
\[
\mathbb{A}=
\begin{bmatrix}
a & b & c\\
\alpha & \beta & \gamma\\
\lambda & \mu & \nu
\end{bmatrix}
\qquad\text{and}\qquad
\mathbb{B}=
\begin{bmatrix}
b_{1}\\
b_{2}\\
b_{3}
\end{bmatrix}.
\]

The algebra $A$ is generated over $\K$ by the elements $x_{1}$, $x_{2}$ and $x_{3}$ subject to the relations
\begin{align}
x_{2}x_{1}-q_{1}x_{1}x_{2} &= a x_{1}+b x_{2}+c x_{3}+b_{1}, \label{eq:2.9}\\
x_{3}x_{1}-q_{2}x_{1}x_{3} &= \alpha x_{1}+\beta x_{2}+\gamma x_{3}+b_{2}, \label{eq:2.10}\\
x_{3}x_{2}-q_{3}x_{2}x_{3} &= \lambda x_{1}+\mu x_{2}+\nu x_{3}+b_{3}. \label{eq:2.11}
\end{align}
The PBW property for this particular class of bi-quadratic algebras is stated below.
\begin{teorema}{\cite[Theorem 1.3]{Bavula2023}}
Suppose that the algebra $A$ is generated over a field $\K$ by elements $x_1$, $x_2$, $x_3$ satisfying relations \eqref{eq:2.9}, \eqref{eq:2.10} and \eqref{eq:2.11}. Then the defining relations are consistent and
\begin{align*}
A=\bigoplus_{\alpha\in\mathbb{N}^3} \K x^\alpha,
\text{ where } x^\alpha=x_1^{\alpha_1}x_2^{\alpha_2}x_3^{\alpha_3},
\end{align*}
if and only if the following conditions hold:
\begin{gather*}
(1-q_3)\alpha = (1-q_2)\mu, \\
(1-q_3)a = (1-q_1)\nu, \\
(1-q_2)b = (1-q_1)\gamma, \\
(1-q_1q_2)\lambda = 0, \\
(q_1-q_3)\beta = 0, \\
(1-q_2q_3)c = 0, \\
((1-q_3)a-\mu)a+(b+q_1\gamma)\lambda-\nu\alpha+(q_1q_2-1)b_3 = 0, \\
(a-\nu)\beta+q_1\gamma\mu-q_3\alpha b+(q_1-q_3)b_2 = 0, \\
a+(q_1-1)\nu\gamma+b\gamma-(\mu+q_3\alpha)c+(1-q_2q_3)b_1 = 0, \\
-(\mu+q_3\alpha)b_1+(a-\nu)b_2+(b+q_1\gamma)b_3 = 0.
\end{gather*}

In addition, if $A=\bigoplus_{\alpha\in\mathbb{N}^3} \K x^\alpha$, with
$x^\alpha=x_1^{\alpha_1}x_2^{\alpha_2}x_3^{\alpha_3}$, then
\[
A=\bigoplus_{\alpha\in\mathbb{N}^3} \K x_\sigma^\alpha
\]
for all $\sigma\in S_3$,where
\[
x_\sigma^\alpha:=x_{\sigma(1)}^{\alpha_1}x_{\sigma(2)}^{\alpha_2}x_{\sigma(3)}^{\alpha_3}.
\]
\end{teorema}

Henceforth, the term «bi-quadratic algebra» means bi-quadratic algebra with PBW basis. An important family of bi-quadratic algebras was introduced by Bavula in \cite{Bavula2024}. These algebras, called \textbf{3-cyclic quantum Weyl algebras}, arose naturally in the author's work on the classification of Harish--Chandra modules over the quantized Lorentz algebra (see \cite{Bavula2024} and references therein).
For $\alpha$, $\beta$, $\gamma \in \K$, the 3-cyclic quantum Weyl algebra $A(\alpha, \beta, \gamma)$ is the $\K$-algebra generated by $x$, $y$, $z$ subject to the defining relations
\begin{equation}\label{relaciones-Aabg}
xy = q^{2}yx + \alpha, \qquad xz = q^{-2}zx + \beta, \qquad yz = q^{2}zy + \gamma.
\end{equation}
The algebra $A(\alpha,\beta,\gamma)$ admits a presentation as an iterated Ore extension of the form $\K[x][y;\sigma_1,\delta_1][z;\sigma_2,\delta_2]$, where $\sigma_1$ is the automorphism of $\K[x]$ defined by $\sigma_1(x) = q^{-2}x$, and $\delta_1$ is the $\sigma_1$-derivation of $\K[x]$ satisfying $\delta_1(x) = -q^{-2}\alpha$. In turn, $\sigma_2$ is the automorphism of $\K[x][y;\sigma_1,\delta_1]$ determined by $\sigma_2(x) = q^{2}x$ and $\sigma_2(y) = q^{-2}y$, and $\delta_2$ is the $\sigma_2$-derivation of $\K[x][y;\sigma_1,\delta_1]$ satisfying $\delta_2(x) = -q^{2}\beta$ and $\delta_2(y) = -q^{-2}\gamma$. In particular, the algebra $A(\alpha,\beta,\gamma)$ is a Noetherian domain of Gelfand--Kirillov dimension $3$.

For each integer a $\geq 1$, define the coefficients
\[
c_{a}:=\frac{1-q^{2a}}{1-q^{2}},\qquad
d_{a}:=\frac{1-q^{-2a}}{1-q^{-2}}.
\]
In order to study the PI property for bi-quadratic algebras, it is necessary to consider the following identities satisfied by these algebras.
\begin{teorema}{\cite[Theorem 2.1]{BeraMandalMukherjeeNandy2024}}\label{identidesbicuadraticas}
Assume that $q^2\neq 1$. The following identities hold in $A(\alpha,\beta,\gamma)$,
for all $a\in\mathbb{Z}^{+}$:
\begin{enumerate}[\rm (i)]
\item $\displaystyle
x^{a}y = q^{2a}yx^{a}
          + c_{a}\,\alpha x^{a-1}.
  $

  \item $\displaystyle
  y^{a}x = q^{-2a}xy^{a}
          - q^{-2}d_{a}\,\alpha y^{a-1}.
  $

  \item $\displaystyle
  x^{a}z = q^{-2a}zx^{a}
          + d_{a}\,\beta x^{a-1}.
  $

  \item  $\displaystyle
  z^{a}x = q^{2a}xz^{a}
          - q^{2}c_{a}\,\beta z^{a-1}.
  $

  \item  $\displaystyle
  y^{a}z = q^{2a}zy^{a}
          + c_{a}\,\gamma y^{a-1}.
  $

  \item $\displaystyle
  z^{a}y = q^{-2a}yz^{a}
          - q^{-2}d_{a}\,\gamma z^{a-1}.
  $
\end{enumerate}
\end{teorema}
\begin{proof}
Each of the identities is established by induction on
$a\in\mathbb{Z}^{+}$. The defining relations of $A(\alpha,\beta,\gamma)$ yield the identities for $a=1$, see  (\ref{relaciones-Aabg}).

\begin{itemize}
\item[(i)] Note that,
\begin{align*}
q^{2a}+c_a
&= q^{2a} + \frac{1-q^{2a}}{1-q^{2}}
 = \frac{q^{2a}(1-q^{2}) + 1 - q^{2a}}{1-q^{2}} \\
&= \frac{1-q^{2(a+1)}}{1-q^{2}}
 = c_{a+1}.
\end{align*}
Hence,
 \[
\begin{aligned}
x^{a+1}y
  &= x(x^{a}y)
   = x\big(q^{2a}yx^{a} + c_{a}\,\alpha x^{a-1}\big) \\
  &= q^{2a}x y x^{a} + c_{a}\,\alpha x^{a}.
\end{aligned}
\]
Using the relation $xy = q^{2}yx + \alpha$, one obtains 
\[
\begin{aligned}
x^{a+1}y
 &= q^{2a}(q^{2}yx + \alpha)x^{a} + c_{a}\,\alpha x^{a} \\
 &= q^{2(a+1)}yx^{a+1}
    + q^{2a}\alpha x^{a} + c_{a}\,\alpha x^{a} \\
 &= q^{2(a+1)}yx^{a+1}
    + (q^{2a}+c_{a})\,\alpha x^{a}\\
 &= q^{2(a+1)}yx^{a+1} + c_{a+1}\,\alpha x^{a}.
\end{aligned}
\]

\item[(ii)] Observe that,
\begin{align*}
q^{-2a}+d_a
&= q^{-2a} + \frac{1-q^{-2a}}{1-q^{-2}}\\
&=  \frac{q^{-2a}(1-q^{-2}) + 1 - q^{-2a}}{1-q^{-2}} \\
&= \frac{1-q^{-2(a+1)}}{1-q^{-2}}
 = d_{a+1}.
\end{align*}
Thus,
\[
\begin{aligned}
y^{a+1}x
  &= y(y^{a}x)
   = y\big(q^{-2a}xy^{a} - q^{-2}d_{a}\,\alpha y^{a-1}\big) \\
  &= q^{-2a}yxy^{a} - q^{-2}d_{a}\,\alpha y^{a}.
\end{aligned}
\]
Using the identity  $yx = q^{-2}xy - q^{-2}\alpha$, the following is obtained
\[
\begin{aligned}
y^{a+1}x
 &= q^{-2a}\big(q^{-2}xy - q^{-2}\alpha\big)y^{a}
    - q^{-2}d_{a}\,\alpha y^{a} \\
 &= q^{-2(a+1)}xy^{a+1}
    - q^{-2(a+1)}\alpha y^{a}
    - q^{-2}d_{a}\,\alpha y^{a} \\
 &= q^{-2(a+1)}xy^{a+1}
    - q^{-2}\big(q^{-2a} + d_{a}\big)\alpha y^{a} \\
 &= q^{-2(a+1)}xy^{a+1}
    - q^{-2}d_{a+1}\,\alpha y^{a}.
\end{aligned}
\]

\item[(iii)] From the equality  $q^{-2a} + d_{a}  = d_{a+1}$ one sees that  
\[
\begin{aligned}
x^{a+1}z
  &= x(x^{a}z)
   = x\big(q^{-2a}zx^{a} + d_{a}\,\beta x^{a-1}\big) \\
  &= q^{-2a}xzx^{a} + d_{a}\,\beta x^{a}.
\end{aligned}
\]

The relation $xz = q^{-2}zx + \beta$, yields 
\[
\begin{aligned}
x^{a+1}z
 &= q^{-2a}\big(q^{-2}zx + \beta\big)x^{a} + d_{a}\,\beta x^{a} \\
 &= q^{-2(a+1)}zx^{a+1}
    + q^{-2a}\beta x^{a}
    + d_{a}\,\beta x^{a} \\
 &= q^{-2(a+1)}zx^{a+1}
    + \big(q^{-2a} + d_{a}\big)\beta x^{a} \\
 &= q^{-2(a+1)}zx^{a+1}
    + d_{a+1}\,\beta x^{a}.
\end{aligned}
\]

\item[(iv)] Using the equality $q^{2a}+c_{a}= c_{a+1}$ established above, the following holds
\[
\begin{aligned}
z^{a+1}x
  &= z(z^{a}x)
   = z\big(q^{2a}xz^{a} - q^{2}c_{a}\,\beta z^{a-1}\big) \\
  &= q^{2a}z x z^{a} - q^{2}c_{a}\,\beta z^{a}.
\end{aligned}
\]
In consequence, 
\[
\begin{aligned}
z^{a+1}x
 &= q^{2a}\big(q^{2}xz - q^{2}\beta\big)z^{a}
    - q^{2}c_{a}\,\beta z^{a} \\
 &= q^{2(a+1)}xz^{a+1}
    - q^{2(a+1)}\beta z^{a}
    - q^{2}c_{a}\,\beta z^{a} \\
 &= q^{2(a+1)}xz^{a+1}
    - q^{2}\big(q^{2a} + c_{a}\big)\beta z^{a} \\
 &= q^{2(a+1)}xz^{a+1}
    - q^{2}c_{a+1}\,\beta z^{a}.
\end{aligned}
\]

\item[(v)] Once again, using the fact that  $q^{2a} + c_{a}  = c_{a+1}$, one obtains
\[
\begin{aligned}
y^{a+1}z
  &= y(y^{a}z)
   = y\big(q^{2a}zy^{a} + c_{a}\,\gamma y^{a-1}\big) \\
  &= q^{2a}yzy^{a} + c_{a}\,\gamma y^{a}.
\end{aligned}
\]
In this way,
\[
\begin{aligned}
y^{a+1}z
 &= q^{2a}\big(q^{2}zy + \gamma\big)y^{a}
    + c_{a}\,\gamma y^{a} \\
 &= q^{2(a+1)}zy^{a+1}
    + q^{2a}\gamma y^{a}
    + c_{a}\,\gamma y^{a} \\
 &= q^{2(a+1)}zy^{a+1}
    + \big(q^{2a} + c_{a}\big)\gamma y^{a} \\
 &= q^{2(a+1)}zy^{a+1}
    + c_{a+1}\,\gamma y^{a}.
\end{aligned}
\]
\item[(vi)] Finally, since $q^{-2a} + d_{a}  = d_{a+1}$, 
\[
\begin{aligned}
z^{a+1}y
  &= z(z^{a}y)
   = z\big(q^{-2a}yz^{a} - q^{-2}d_{a}\,\gamma z^{a-1}\big) \\
  &= q^{-2a}zyz^{a} - q^{-2}d_{a}\,\gamma z^{a}.
\end{aligned}
\]

Therefore, 
\[
\begin{aligned}
z^{a+1}y
 &= q^{-2a}\big(q^{-2}yz - q^{-2}\gamma\big)z^{a}
    - q^{-2}d_{a}\,\gamma z^{a} \\
 &= q^{-2(a+1)}yz^{a+1}
    - q^{-2(a+1)}\gamma z^{a}
    - q^{-2}d_{a}\,\gamma z^{a} \\
 &= q^{-2(a+1)}yz^{a+1}
    - q^{-2}\big(q^{-2a} + d_{a}\big)\gamma z^{a} \\
 &= q^{-2(a+1)}yz^{a+1}
    - q^{-2}d_{a+1}\,\gamma z^{a}.
\end{aligned}
\]
\end{itemize}
\end{proof}

\begin{corolario}{\cite[Corollary 2.2]{BeraMandalMukherjeeNandy2024}}\label{cor:2.2}
Assume that $q^2\neq 1$. If $q^{2}$ is a primitive $l$-th root of unity, then $x^{l}$, $y^{l}$ and $z^{l}$ are central elements of $A(\alpha,\beta,\gamma)$.
\end{corolario}

\begin{proof}
Suppose that $q^{2}$ is a primitive $l$-th root of unity.
In particular,
\[
1-q^{2l}=0,
\qquad
1-q^{-2l}=0.
\] 
For $a=l$, the identities established in Theorem \ref{identidesbicuadraticas} yield:
\begin{align*}
x^{l}y
  &= q^{2l}yx^{l}
    + c_a\,\alpha x^{l-1}\\
  &= 1\cdot yx^{l} + 0\cdot x^{l-1}\\
  &= yx^{l},\\
  x^{l}z
  &= q^{-2l}zx^{l}
    + d_a\,\beta x^{l-1}\\
  &= 1\cdot zx^{l} + 0\cdot x^{l-1}\\
  &= zx^{l},
\end{align*}
thus $x^{l}$ commutes with the generators $x$, $y$, $z$, and hence $x^{l} \in Z(A(\alpha,\beta,\gamma))$.

By analogous arguments, the following equalities are obtained:
\begin{align*}
y^{l}x
  &= q^{-2l}xy^{l}
    - q^{-2}d_a\,\alpha y^{l-1}\\
  &= 1\cdot xy^{l} - q^{-2}\cdot 0\cdot y^{l-1}\\
  &= xy^{l},\\
y^{l}z
  &= q^{2l}zy^{l}
    + c_a\,\gamma y^{l-1}\\
  &= 1\cdot zy^{l} + 0\cdot y^{l-1}\\
  &= zy^{l},
  \end{align*}
  which allows one to conclude that  $y^{l}\in Z(A(\alpha,\beta,\gamma))$.
Finally, for $a=l$ in identities (iv) and (v) of the Theorem \ref{identidesbicuadraticas}, one obtains
\[
z^{l}x
  = q^{2l}xz^{l}
    - q^{2}c_l\,\beta z^{l-1}
  = 1\cdot xz^{l} - q^{2}\cdot 0\cdot z^{l-1}
  = xz^{l},
\]
and
\[
z^{l}y
  = q^{-2l}yz^{l}
    - q^{-2}\,\gamma z^{l-1}
  = 1\cdot yz^{l} - q^{-2}\cdot 0\cdot z^{l-1}
  = yz^{l}.
\]
Thus $z^{l}$ commutes with $x$ and $y$, and hence $z^{l}\in Z(A(\alpha,\beta,\gamma))$.

\end{proof}

\begin{lema}\label{lema:ez-q-2ze}
Assume that $q^2\neq 1$. Let $A(\alpha,\beta,\gamma)$ be the 3-cyclic quantum Weyl algebra and \[
e := xz - \frac{q^{2}\beta}{q^{2}-1}.
\]
Then the following identity holds:
\[
ez = q^{-2}ze.
\]
\end{lema}
\begin{proof}
First of all, 
\[
\begin{aligned}
ez
  &= \Bigl(xz - \frac{q^{2}\beta}{q^{2}-1}\Bigr)z 
  = xz^{2} - \frac{q^{2}\beta}{q^{2}-1}\,z.
\end{aligned}
\] 
Using the relation $xz = q^{-2}zx + \beta$, one obtains
\[
\begin{aligned}
xz^{2}
  &= (xz)z
   = (q^{-2}zx + \beta)z \\
  &= q^{-2}zxz + \beta z.
\end{aligned}
\]
Substituting into the expression for $ez$, it follows that
\[
\begin{aligned}
ez
 &= q^{-2}zxz + \beta z
    - \frac{q^{2}\beta}{q^{2}-1}\,z \\
 &= q^{-2}zxz + \left(1 - \frac{q^{2}}{q^{2}-1}\right)\beta z \\
 &= q^{-2}zxz - \frac{\beta}{q^{2}-1}\,z.
\end{aligned}
\] On the other hand,
\[
\begin{aligned}
q^{-2}ze
  &= q^{-2}z\Bigl(xz - \frac{q^{2}\beta}{q^{2}-1}\Bigr) \\
  &= q^{-2}zxz - q^{-2}\frac{q^{2}\beta}{q^{2}-1}\,z \\
  &= q^{-2}zxz - \frac{\beta}{q^{2}-1}\,z.
\end{aligned}
\]
It is thus concluded that
\[
ez = q^{-2}ze,
\]
\end{proof}

\begin{teorema}{\cite[Theorem 3.2]{BeraMandalMukherjeeNandy2024}}\label{thm:3.2}
Assume that $q^2\neq 1$. The algebra $A(\alpha,\beta,\gamma)$ is a
PI algebra if and only if $q^2$ is a root of unity.
\end{teorema}

\begin{proof}
Suppose that $q^{2}$ is a primitive 
$l$-th root of unity. Then, by Corollary \ref{cor:2.2}, 
\[
 B=\K[x^{l},y^{l},z^{l}] \subseteq Z(A(\alpha,\beta,\gamma))
\] is a central subalgebra of $A(\alpha,\beta,\gamma)$. This implies that $A(\alpha,\beta,\gamma)$   is finitely generated as an 
B-module. Hence, by \cite[\S 13.1.13, Corollary]{mcconnell}, $A(\alpha,\beta,\gamma)$ satisfies the PI property.

Conversely, suppose that  $A(\alpha,\beta,\gamma)$  satisfies the PI property. Then every $\K$-subalgebra of  $A(\alpha,\beta,\gamma)$ satisfies the PI property. In particular, using the relation established in Lemma \ref{lema:ez-q-2ze}, the subalgebra $\K\langle z,e \mid ez = q^{-2}ze \rangle$  isomorphic to the quantum plane with parameter  $q^{-2}$ and by  \cite[Section \S 7.1]{DeConciniProcesi1993},  one concludes that $q^{2}$ is a root of unity.
\end{proof}


\section{Down-Up algebras}\label{DownUpalgebras}
Down-up algebras have been introduced by Benkart and Roby in \cite{benkart} motivated by the study of posets. Given a field $\K$ and constants $\alpha$, $\beta$, $\gamma$ in $\K$, the \emph{Down-Up algebra} $A=A(\alpha,\beta,\gamma)$ is the associative algebra generated over $\K$ by $u$ and $d$, subject to the defining relations:
\begin{align*}
du^2&=\alpha udu+\beta u^2d+\gamma u\\
d^2u&=\alpha dud+\beta ud^2+\gamma d.
\end{align*}
As known examples of Down-Up algebras, we can mention $A(2,-1,0)$ that turns out to be isomorphic to the enveloping algebra of the Heisenberg Lie algebra of dimension 3; for the case  where $\gamma\neq 0$, the algebra $A(2,-1,\gamma)$ is isomorphic to the enveloping algebra of $\mathrm{sl}_2(\K)$. For another interesting example, consider the quantized enveloping algebra $U_q(\mathrm{sl}_3(\K))$ with generators $E_i$, $F_i$, $K^{\pm 1}$, $i=1$, $2$ and a non-zero scalar $q$ in $\K$; the subalgebra of $U_q^{+}(\mathrm{sl}_3(\K))$ generated by $E_1$, $E_2$ is the Down-Up algebra $A([2]_{q},-1,0)$, where $[n]_{q}=\frac{q^{n}-q^{-n}}{q-q^{-1}}$. For $\gamma\neq 0$, the Down-Up algebra $A(0,1,\gamma)$ is isomorphic to the  enveloping algebra of the Lie superalgebra $\mathrm{osp}(1,2)$.\\
\\
Kirkman, Musson and Passman proved in \cite{Kirkman}  that $A(\alpha,\beta,\gamma)$ is a noetherian algebra if and only if the parameter $\beta$ is non-zero; the latter is equivalent to saying that $A(\alpha,\beta,\gamma)$ is a domain. Furthermore, for Down-Up algebras, the Krull, Gelfand-Kirillov, and global dimensions have already been computed, see \cite{Bavula4}, \cite{benkart} and \cite{Kirkman}. Additionally, their representation theory, Hochschild homology and cohomology, as well as several homological and ring theoretical properties have also been studied  (e.g., \cite{carvalho}, \cite{carvalho2}, \cite{solotar}, \cite{Jordan}, \cite{gallego}).

In what follows it will be assumed that the roots of the polynomial $t^2-\alpha t-\beta$ are in the field $\K$. 
\begin{teorema}\cite[\S 4.4]{Kirkman}
If $A=A(\alpha, \beta, \gamma)$ is a Down-Up algebra over a field $\K$, then the following statements are equivalent
\begin{enumerate}
    \item [\rm (1)] $\beta\neq 0$.
\item [\rm (2)]  $A$ is right (or left) noetherian.
\item [\rm (3)] $A$ is a domain.
\item [\rm (4)] $\K[ud, du]$ is a polynomial ring in two generators. 
\end{enumerate}
Moreover, in these conditions $A$ is an Auslander-regular algebra with global dimension 3.
\end{teorema}
Given a noetherian Down-Up algebra $A(\alpha,\beta,\gamma)$,  let $\lambda$ and $\mu$ be the roots of $t^2-\alpha t-\beta$, so that $\alpha=\lambda+\mu$ and $\beta=-\lambda\mu$. Since $\beta$ is non-zero, it follows that $\lambda$ and $\mu$ are both non-zero. Furthermore, if $\gamma\neq 0$, there is an isomorphism $A(\alpha,\beta,\gamma)\cong A(\alpha,\beta,1)$, see \cite[Lemma 4.1 (ii)]{carvalho}, so we can assume $\gamma=1$ without loss of generality. Under these conditions, the multiplication rules in $A$ are given by:
\begin{align*}
[d,[d,u]_{\lambda}]_{\mu}=d, \quad\quad  \text{\, and\, } \quad\quad [[d,u]_{\lambda},u]_{\mu}=u,
\end{align*}
where $[a,b]_{\eta}$ denotes the expression $ab-\eta ba$.\\
\\
Furthermore, recall that given a $\K$-algebra $R$, an automorphism $\sigma$ of $R$ and a central element $a$ of $R$, the \textit{\-ge\-ne\-ra\-li\-zed Weyl algebra} $R(\sigma,a)$ is defined as the algebra generated by $X$ and $Y$ subject to the relations: $YX=a$, $XY=\sigma(a)$, $Yb=\sigma^{-1}(b)Y$ and $Xb=\sigma(b)X$ for all $b\in R$.\\
\\
In \cite[\S 2.2]{Kirkman} it is shown that an arbitrary noetherian Down-Up algebra is isomorphic to a \-ge\-ne\-ra\-li\-zed Weyl algebra. Specifically,  taking $R=\K[x,y]$ and $\phi$ the automorphism of $R$ defined by 
\begin{align}\label{autodownup}
 \phi(x)=y, \text{\qquad and \qquad}  \phi(y)=\alpha y+\beta x+\gamma 
\end{align}
and $a=x$, the algebra $A(\alpha,\beta,\gamma)$ turns out to be  isomorphic to $R(\phi,x)$ under the algebra isomorphism $\varphi$ sending $X$ to $d$ and $Y$ to $u$; in particular, $x$ and $y$ correspond to $ud$ and $du$, respectively. Using this fact, Kulkarni in \cite{Kulkarni} computed the center of the Down-Up algebra $A(\alpha,\beta,\gamma)$ using the characterization for the center of an arbitrary generalized Weyl algebra $R(\sigma,a)$ introduced by him. Explicitly, he obtained the following facts:
\begin{teorema}\cite[Proposition 2.0.1, Corollary 2.0.2]{Kulkarni} \label{centergwa} Let $R$ be a commutative domain. The center of the \-ge\-ne\-ra\-li\-zed Weyl algebra $R(\sigma,a)$ is non-trivial  if, and only if, one of the following statements holds:
\begin{enumerate}
    \item [\rm (i)] There exists a non-trivial element $r\in R$ such that $\sigma(r)=r$.
    \item [\rm (ii)] $\sigma^m=\mathrm{id}_{R}$ for some positive integer $m$.
\end{enumerate}
In these conditions, if $m\in \Z^+$ satisfies $\sigma^m=\mathrm{id}_{R}$ (if there is no such integer $m$, take $m=0$) the center of $R(\sigma, a)$ is generated by $\{X^m, Y^m\}$   and all the elements of $R$ which are fixed by $\sigma$; that is, $Z(R(\sigma,a))=R^{\sigma}[X^m,Y^m]$, where $R^{\sigma}:=\{r\in R\mid \sigma(r)=r\}$.
\end{teorema}
Therefore, if $\sigma$ has finite order, then $R(\sigma,a)$ is generated by $1, X,\ldots, X^{m-1}$, $Y,\ldots, Y^{m-1}$ as module over its center. On the contrary, if $\sigma$ has infinite order, then  $Z(R(\sigma, a))=R^{\sigma}$. In these conditions,  the algebra $R(\sigma,a)$ cannot be finitely generated on $Z(R(\sigma,a))$ as is proved below.
\begin{lema}\label{fingencenalg}
Let $R(\sigma,a)$ be a generalized Weyl algebra. Then
$R(\sigma,a)$ is finitely generated as a module over a central
subalgebra if and only if the automorphism $\sigma$ has finite order.
\begin{proof}
For $A:=R(\sigma,a)$, consider the canonical $\Z$-grading of $R(\sigma,a)$ determined by
\begin{align*}
\deg(R)=0,\qquad \deg(X)=1 \text{\qquad and \qquad} \deg(Y)=-1.
\end{align*}
Under these conditions, it follows that $A=\bigoplus_{n\in\Z}A_n$, where
\begin{align*}
 A_0=R,\qquad A_n=RX^{n}, \text{\qquad and \qquad} A_{-n}=RY^n,  
\end{align*}
for each $n\in \Z^+$. If $\sigma$ has not finite order, then $Z(A)=R^{\sigma}\subseteq R=A_0$, so if $C$ is a central subalgebra of $A$, then $C\subseteq Z(A)$; that is, every element of $C$ has degree $0$. Therefore, assuming that $A$ is finitely generated as --say-- left module over  one of such $C$, one would have 
\begin{align*}
   A=\sum_{i=1}^t C a_i, 
\end{align*}
for some $a_1,\ldots, a_t\in A$. Each of these $a_i$ is written as
\begin{align*}
  a_i=  \sum_{k=l_i}^{m_i} a_{k}^{(i)}
\end{align*}
for certain integers $l_i\leq m_i$, 
with $a_{k}^{(i)}$ a homogeneous element of degree $k$. Let $l:=\mathrm{min}\{l_i\}_{i=1}^{t}$ and $m:=\mathrm{max}\{m_i\}_{i=1}^t$. Since $A_0 A_k\subseteq A_k$, it follows that every element of the sum $\sum _{i=1}^{t}Ca_i$ lives only in degrees between $l$ and $m$; that is, 
$\sum_{i=1}^t Ca_i\subseteq \bigoplus _{k=l}^{m}A_k$. Thus, if $n\geq m+1$ (or $n\leq l-1$) and $0\neq a\in A_n$, then $a\notin \sum_{i=1}^t Ca_i$, which contradicts the finite generation hypothesis on $A$.
\end{proof}
\end{lema}
Indeed, given an arbitrary Down-Up algebra $A=A(\alpha,\beta,\gamma)$, and by means the algebra isomorphism $A\cong \K[x,y](\phi,x)$ established before, Kulkarni described completely the center of $A$. When such an automorphism $\phi$ has finite order, Kullkanri says that the noetherian Down-Up algebra $A(\alpha,\beta,\gamma)$ is a \textbf{Down-Up algebra at roots of unity}. Recall that throughout this section we assume that $\operatorname{char}(\K)=0$. This assumption will be used in the arguments involving finite-order
automorphisms and Jordan blocks. In particular, it ensures that coefficients such as $n\mu^{n-1}\omega_1$ do not vanish merely because the characteristic of the base field divides $n$.
\begin{teorema}\cite[Theorem 4.0.2]{Kulkarni}\label{centerdownup}
Let $\lambda$ and $\mu$ be the roots of the quadratic polynomial $t^2-\alpha t-\beta$. The center of the Down-Up algebra $A(\alpha,\beta,\gamma)$ is non-trivial if, and only if, one of the following conditions is satisfied:
\begin{enumerate}
\item [\rm (1)] $\lambda^i\mu^j=1$ for some $i$, $j\in \Z^+$.
\item [\rm (2)] $\mu$ is a primitive $m$-th root of unity, with $\mu\neq 1$.
\item [\rm (3)] $\lambda=1$ and $\gamma=0$.
\item[\rm (4)] $\mu=\lambda=1$.
\end{enumerate}
\begin{proof}
 The goal is to establish the conditions for which the automorphism $\phi$  has a non-trivial fixed set in $R=\K[x,y]$  or $\phi^m=\mathrm{id}_{R}$. To begin with, suppose $\alpha^2+4\beta\neq 0$ and $\alpha+\beta\neq 1$. These conditions imply that $\mu\neq \lambda$ and neither of them is equal to $1$.  Case 1 in \cite[pg. 3]{carvalho} ensures that 
\begin{align}\label{omegas}
  \omega_i:=\beta(r_i-1)x+r_i(r_i-1)y+\gamma r_i,   
 \end{align}
 for $i\in\{1,2\}$, with $r_1=\mu$ and $r_2=\lambda$,  are such that $\phi(\omega_i)=r_i\omega_i$. Since the matrix $\begin{pmatrix} \beta(\lambda-1) & \lambda(\lambda-1)\\ \beta(\mu-1) & \mu(\mu-1) \end{pmatrix}$ is invertible, it is possible to define a change of coordinates over $R$ by means $x\mapsto \omega_1$ and $y\mapsto \omega_2$, so that $\K[x,y]=\K[\omega_1,\omega_2]$. In particular, 
 \begin{align*}
\phi(\omega_1^i\omega_2^j)=\mu^i\lambda^j \omega_1^i\omega_2^j.
 \end{align*}
Since $\phi$ preserves grading with respect these new coordinates, it will have a fixed element if, and only if, $\mu^i\lambda^j=1$ for some $i$, $j\in \Z^+$. Moreover, there is a $m\in\Z^+$ such that $\phi^m=\mathrm{id}_{R}$ if, and only if, $\mu^m=\lambda^m=1$. Now consider the case $\alpha^2+4\beta\neq 0$ and $\alpha+\beta=1$. In these conditions $\mu\neq \lambda$, $\lambda=1$ and $\alpha\neq 2$. Taking
\begin{align*}
  \omega_1&=\beta x +y\\
  \omega_2&=-x+y+\frac{\gamma}{\alpha-2},
\end{align*}
it is valid that $\phi(\omega_1)=\omega_1+\gamma$ and $\phi(\omega_2)=\mu\omega_2$. If $\gamma\neq 0$ then the automorphism $\phi$ cannot be or finite order. Thus, the center is non-trivial if, and only if, $\mu^{i}=1$ for some $i\in \Z^+$.  Certainly, if $\gamma=0$, the element $\omega_1$ would be a non-trivial element in the center of $A(\alpha,\beta,0)$. Now, if $\alpha^2+4\beta=0$ and $\alpha+\beta=1$, then roots $\mu$ and $\lambda$ are both equal to $1$. So, $t^2-\alpha t-\beta=(t-\mu)^2$, $\alpha=2$, $\beta=-1$ and the linear part or $\phi$ has a single eigenvalue $\mu$. For this case, the linear polynomials
\begin{align*}
 \omega_1&=-x+y+\gamma,\\
 \omega_2&=y
\end{align*}
satisfy $\phi(\omega_1)=\omega_1+\gamma$ and $\phi(\omega_2)=\omega_1+\omega_2$. The latter implies that $\phi$ never has finite order. So, the only possible way to have a non-trivial center is if exists  some non-scalar element of $R$ fixed by $\phi$. Note that if $\gamma=0$, such an element is $\omega_1=-x+y$. On the contrary, for $\gamma\neq 0$, it is well known that $A(2,-1,\gamma)\cong \mathcal{U}(\mathrm{sl}_2(\K))$, the universal enveloping algebra of $\mathrm{sl}_2(\K)$, and the center of the last one is non-trivial (see, e.g., \cite[Example 13.10]{Etingof2024}). Finally, for $\alpha^2+4\beta=0$ with $\alpha+\beta\neq 1$, it follows that $\mu\neq 1$. Given these conditions, the linear polynomials
\begin{align*}
 \omega_1&=(2\beta +\alpha)x+(\alpha-2)y+2\gamma,\\
 \omega_2&=2(-x+y)
\end{align*}
have the particularity that $\phi(\omega_1)=\left(\frac{\alpha}{2}\right)\omega_1$ and $\phi(\omega_2)=\omega_1+\left(\frac{\alpha}{2}\right)\omega_2$. Note that $\frac{\alpha}{2}=\mu$. In these new coordinates, the action of $\phi$ is homogeneous and, since the only eigenvalue of  $\phi$ on $R_n:=\langle\{\omega_1^i\omega_2^j\mid i+j=n\}\rangle_{\K}$ is $\mu^n$, then $\phi$ can fix some element of $R$ if, and only if, $\mu^n=1$ for some $n\in \Z^+$. Note that under these circumstances over $\alpha$ and $\beta$,  the linear part of $\phi$ in degree $1$ --after change of coordinates-- is represented by $L=\begin{pmatrix}
 \lambda& 1\\ 0 &\lambda    \end{pmatrix}$ and, therefore, there is no $m\in \Z^+$ with $L^m=I$; that is, $\phi$ has not finite order. 
\end{proof}
\end{teorema}
The latter theorem allows the following explicit description of $Z(A(\alpha,\beta, \gamma))$.
\begin{teorema}\cite[Theorem 4.0.3, Theorem 4.0.4]{Kulkarni}
On the same conditions as before,  suppose that $\alpha^2+4\beta\neq 0$ and $\lambda^i\mu^j=1$, for some $i$, $j\in\Z^+$. If $\lambda$ and $\mu$ are roots of unity, let $m:=\mathrm{lcm}\{i,j\}$. Otherwise let $m=0$. Then, the center of  $A(\alpha,\beta,\gamma)$ is the subalgebra generated by $\{U^m,D^m\}$ and  $\{\omega_1^i\omega_2^j\mid \lambda^i\mu^j=1\}$, with $\omega_1$ and $\omega_2$ as in (\ref{omegas}). Furthermore, if one of the following holds
\begin{enumerate}[\rm (1)]
\item $\alpha^2+4\beta=0$ and $\mu$ is a non-trivial $m$-th root of unity.
\item  $\alpha^2+4\beta\neq 0$, $\lambda=1$, $\mu$ is a primitive $m$-th root of unity and $\gamma\neq 0$.
\end{enumerate}
Then the center of $A(\alpha,\beta,\gamma)$ is the polynomial ring $\K[\omega^m]$ where $\omega$ is given by
\begin{align*}
     \omega&=(2\beta +\alpha)x+(\alpha-2)y+2\gamma
     \end{align*}
in the first case, and 
\begin{align*}
 \omega&=2(-x+y)
\end{align*}
in the second case.
\end{teorema}
\begin{remark} It seems  that the case when $\alpha^2+4\beta\neq 0$, $\alpha+\beta=1$ and $\gamma=0$ is not covered by the last theorem. In this case $\mu\neq \lambda$ and $\lambda =1$ and,  as was shown in the proof of the Theorem \ref{centerdownup}, $\omega_1=\beta x+y=\beta ud+du$ is fixed by $\phi$. Thus $\K[\omega_1]\subseteq R^{\phi}$. If $\mu$ is not a root of unity, then the equality $\K[\omega_1]=R^{\phi}=Z(A)$ holds. On the other side, if $\mu$ is a $m$-th root of unity, then $\phi$ has finite order, so $R^{\phi}=\K[\omega_1, \omega_2^m]$ and Theorem \ref{centergwa} ensures that $Z(A)=\K[\omega_1, \omega_2^m,u^m,d^m]$. Thus, this completes the description of the center in this particular
case. However, in view of Theorem \ref{teorema8.7}, the PI property is still characterized by
the finite order of $\phi$.
\end{remark}
Regarding the PI property on these algebras, in \cite[\S 2.4 Theorem]{carvalho3} it is proved the following characterization. Before, remember that a prime ring $A$ is called \textbf{right bounded} if every essential right ideal contains a nonzero ideal. Moreover, $A$ is said to be \textbf{right completely bounded} if $A/P$ is right bounded for every prime ideal $P$. Finally, $A$ is called a \textbf{right FBN ring} if it is right completely bounded and right Noetherian. Analogously, one defines a \textbf{left bounded ring}, \textbf{left completely bounded} and \textbf{left FBN}. 
\begin{teorema} \label{teorema8.7}
 The following statements are equivalent for a Noetherian Down-Up algebra $A = A(\alpha, \beta, \gamma)$:   
 \begin{enumerate}
     \item [\rm (1)] $A$ is a Down-Up algebra at roots of unity (i.e., $\phi$ has finite order);
     \item [\rm (2)] $A$ is a finitely generated module  over a central subalgebra;
     \item [\rm (3)] $A$ is a PI algebra;
     \item [\rm (4)] $A$ is a FBN algebra;
     \item [\rm (5)] The roots of the polynomial $t^2-\alpha t-\beta$ are distinct roots of unity. If $\gamma\neq 0$, these roots are also different from $1$.
 \end{enumerate}
\begin{proof}
The equivalence $(1)\Leftrightarrow (2)$ is ensured by Lemma \ref{fingencenalg}. The implication $(2)\Rightarrow (3)$ is a consequence of the fact that if $S$ is a PI  ring and $T$ an extension of $S$ with $T_S$ finitely generated, then $T$ is a PI ring (see e.g. \cite[\S 13.4.9]{mcconnell}). On the other hand,  since any noetherian prime algebra satisfying a polynomial identity is fully bounded noetherian (see, for instance, \cite[\S 13.6.6]{mcconnell}), then $(3)\Rightarrow (4)$ holds.\\
Now, suppose that $A$ is a FBN algebra. Under these conditions it will proved that the automorphism $\phi$ in (\ref{autodownup}) has finite order. To begin with, it is a well-known fact that any noetherian Down-Up algebra $A(\alpha,\beta,\gamma)$ can be  embedded in to a skew Laurent polynomial ring, see \cite[\S 2.1]{Kirkman}. Specifically, for $R=\K[x,y]$ and $\sigma$ the automorphism of $R$ defined  by $\sigma(x)=y$ and $\sigma(y)=\alpha y+\beta x+\gamma$ (compare with (\ref{autodownup})), if $S:=R[z,z^{-1};\sigma]$ --the skew group ring of the infinite cyclic group $\langle z\rangle$ over $R$, with $rz=z\sigma(r)$, for all $r\in R$-- then the elements $D=z^{-1}$ and $U=xz$ in $S$ are such that $UD=x$, $DU=y$ and  
\begin{align*}
D^2U&=D(DU)=z^{-1}y=\sigma(y)z^{-1}\\
&=(\alpha y+\beta x+\gamma)=\alpha DUD+\beta UD^2+\gamma D;\\
DU^2&=(DU)U=yxz=xz\sigma(y)\\
&=U(\alpha y+\beta x+\gamma)=\alpha UDU+\beta U^2D+\gamma U.
\end{align*}
As a result, there exists an algebra morphism $f:A(\alpha,\beta,\gamma)\to S$ given by $f(d)=D$ and $f(u)=U$. It can be proved that this $f$ is a monomorphism. Note that $S$, as right $R$-module is free and one of its bases is $\{z^n\mid n\in \Z\}$. In \cite[Lemma 1.2]{carvalho} it was proved that $\{d^n\mid n\in \N\}$ is a Ore set of $A$. Even more, the ring $S$ turns out be isomorphic to the localization $A_d$ of $A$ by this Ore set. However, localizations of a FBN ring are also FBN rings (see, for instance, \cite[Lemma 2.1]{Bell}). Therefore $A_d$ and $S$ are fully bounded noetherian. In \cite[Proposition 4.1.12]{carvalhotesis} is proved that, in this framework, the automorphism $\sigma$ has finite order. 
But $\sigma=\phi$, where $\phi$ is the automorphism in (\ref{autodownup}) that allows $A$ to be represented as a generalized weyl algebra. This proves the implication $(4)\Rightarrow (1)$.\\
Finally, for proving $(5)\Leftrightarrow (1)$, let $\mu$, $\lambda$ the roots of $t^2-\alpha t-\beta$ and $\sigma$ the automorphism over $R=\K[x,y]$ considered above in the definition of $S=R[z,z^{-1}; \sigma]$. Let $V$ be the vector $\K$-space spanned by $1$, $x$ and $y$. Then $V$ is stabilized by $\sigma$. Using once again those vectors $\omega_1$ and $\omega_2$ found in \cite[\S 1.4]{carvalho} --which together $1$ form a Jordan basis of $\sigma$ on $V$-- the following cases must be considered: if $\lambda\neq \mu$ and $\lambda$, $\mu\neq 1$, then there exists such a basis such that $\sigma(\omega_1)=\mu\omega_1$ and $\sigma(\omega_2)=\lambda\omega_2$. In this case was showed that $\sigma$ has finite order if, and only if, both $\mu$ and $\lambda$ are roots of unity.\\
If $\lambda=1$ and $\mu\neq 1$, then there exists a basis such that $\sigma(\omega_1)=\omega_1+\gamma$ and $\sigma(\omega_2)=mu\omega_2$. Hence $\sigma$ has finite order if, and only if, $\gamma=0$ and $\mu$ is a root of unity.\\
If $\lambda=\mu$ but different from $1$, then there exists a basis satisfying  $\sigma(\omega_1)=\mu\omega_1$ and $\sigma(\omega_2)=\mu\omega_2+\omega_1$. In these conditions,  for any $n\in \Z^+$ it follows that $\sigma^n(\omega_1)=\mu^n\omega_1$ and $\sigma^n(\omega_2)=\mu^n\omega_2+n\mu^{n-1}\omega_1$ and $\sigma$ cannot have finite order. Analogously, if $\lambda=\mu=1$, it can be showed that the exists a basis such that $\sigma(\omega_1)=\omega_1+\gamma$ and $\sigma(\omega_2)=\omega_1+\omega_2$, implying that $\sigma$ will not have finite order.
\end{proof}
\end{teorema}

\section{Algebra \texorpdfstring{$B_q(f)$}{Bq(f)}}\label{$B_q(f)$}
With the aim of exhibiting a family of algebras whose Ozone group is trivial, Gaddis and Yee recently defined the algebra
$B_q(f)$ in \cite{GaddisYee2025}. To this end, the authors computed the center of these algebras and showed that, when $q$ is a root of unity satisfying certain conditions relative to the polynomial $f$,  certain elements turn out to be central. Recall that, throughout this section, unless otherwise explicitly stated, the base field $\K$ is assumed to have characteristic zero. The positive characteristic case is considered only in the specific results where this hypothesis is explicitly indicated. For $q\in \K^{*}$ and $f\in \K[t]$, the algebra $B_q(f)$ is the free algebra generated by $u$, $v$, $w$ over $\K$, subject to the relations
\[
uv=qvu,\qquad
wu=quw+f(v),\qquad
wv=q^{-1}vw+f(u).
\]

\begin{lema}{\cite[Lemma 2.3]{GaddisYee2025}}\label{B_q(f)ore}
Let $f\in \K[t]$ be a polynomial. Then $B_q(f)$ is the Ore extension
\begin{align*}
B_q(f)=\K_q[u,v][w;\sigma,\delta],
\end{align*}
where $\K_q[u,v]=\K\langle u,v\mid uv-qvu\rangle$ and
$ \sigma(u)=qu$, $\sigma(v)=q^{-1}v$, $\delta(u)=f(v)$, $\delta(v)=f(u)$.
\end{lema}

\begin{proof}
The idea is that $\K_q[u,v]$ is presented by generators $u,v$ and a
single relation $r:=uv-qvu=0$. More precisely, let
\[
F:=\K\langle u,v\rangle
\qquad\text{and}\qquad
I:=\langle r\rangle=\langle uv-qvu\rangle,
\]
so that $\K_q[u,v]=F/I$. Therefore, for a $\K$-linear map $\delta$
defined on the generators to extend correctly to the entire quotient
$\K_q[u,v]=\K\langle u,v\rangle/\langle r\rangle$ as a
$\sigma$-derivation, one must check that the ideal generated by the
defining relation is stable under the appropriate extension of
$\sigma$ and $\delta$; that is,
\[
\sigma(I)\subseteq I
\qquad\text{and}\qquad
\delta(I)\subseteq I.
\]
Since $I$ is generated by $r$, it is enough to verify that
$\sigma(r)\in I$ and $\delta(r)\in I$, equivalently, that
$\sigma(r)=0$ and $\delta(r)=0$ in the quotient. Thus, computing in
the free algebra and passing to the quotient,
\begin{align*}
\sigma(uv-qvu)&=\sigma(u)\sigma(v)-q\sigma(v)\sigma(u)\\
&=(qu)(q^{-1}v)-q(q^{-1}v)(qu)\\
&=uv-qvu=0;
\end{align*}
hence $\sigma$ is well defined on $\K_q[u,v]$.

Writing $\delta(u)=f(v)$ and $\delta(v)=f(u)$, the following is obtained:
\[
\delta(uv)=\sigma(u)f(u)+f(v)v = q\,u\,f(u)+f(v)v;
\]
and
\[
\delta(vu)=\sigma(v)f(v)+f(u)u = q^{-1}v\,f(v)+f(u)u.
\]
Therefore,
\begin{align*}
\delta(uv-qvu)
&=\delta(uv)-q\,\delta(vu)\\
&=\big(q\,u\,f(u)+f(v)v\big)
-q\big(q^{-1}v\,f(v)+f(u)u\big)\\
&=\underbrace{q\,u\,f(u)-q\,f(u)u}_{=\,0}
\;+\;
\underbrace{f(v)v-v\,f(v)}_{=\,0} =0.
\end{align*}
The equalities marked as zero follow from the fact that $u$ commutes
with every polynomial in $u$, and similarly $v$ commutes with every
polynomial in $v$. Hence $\delta(I)\subseteq I$, and therefore
$\delta$ descends to a well-defined $\sigma$-derivation on
$\K_q[u,v]$.
\end{proof}

\begin{proposicion}\label{basepbw} The set  
$
\mathcal{B}:=\{\,v^{\,i}u^{\,j}w^{\,k}\mid i,j,k\in\mathbb{N}\}
$
is a $\K$-basis for $B_q(f)$.
\end{proposicion}

\begin{proof}
By Lemma \ref{B_q(f)ore}, the algebra $B_q(f)$  is the Ore extension $\K_q[u,v][w;\sigma,\delta]$, hence, as a $\K$-vector space,  $\K_q[u,v][w;\sigma,\delta]$ is a free left $\K_q[u,v]$-module with basis $\{1,w,w^2,\dots\}$; so that every element can be written uniquely in the form $\sum_{k\ge 0} a_k\,w^k, a_k\in \K_q[u,v].
$ On the other hand, the quantum plane 
 $\K_q[u,v]=K\langle u,v\mid uv=qvu\rangle$
has PBW basis $\{\,v^{\,i}u^{\,j}\mid  i,j\in\mathbb{N}\,\}$. Thus every element of $B_q(f)$ admits a unique write as a finite linear combination of monomials $v^{i}u^{j}w^{k}$. As a consequence, the set $\mathcal{B}$ is a $\K$-basis of $B_q(f)$.
\end{proof}

For a positive integer $k$, setting $p=1$ in (\ref{pqdef}) and $q\neq 1$, the following $q$-number is obtained
\[
[k]_q=\frac{q^k-1}{q-1}=1+q+\cdots+q^{k-1}.
\]
In addition, if $q=1$ then
\[
[k]_1=\underbrace{1+1+\cdots+1}_{k\ \text{times}}=k.
\]
Moreover, the following holds 
\begin{align*}
[k+1]_q
&= (1+q+\cdots+q^{k-1})+q^k
= [k]_q + q^k \\
&= 1 + q(1+q+\cdots+q^{k-1})
= 1 + q[k]_q 
\end{align*}

\begin{lema} If $q\neq 1$ is a primitive $n$-th root of unity, then
\[
[k]_q=0 \quad \text{if and only if} \quad n\mid k.
\]
\end{lema}
\begin{proof}
 Since $q\neq 1$, it follows that,
\begin{align*}
[k]_q=0
&\Longleftrightarrow
\frac{q^k-1}{q-1}=0
\Longleftrightarrow
q^k-1=0 \\
&\Longleftrightarrow
q^k=1.
\end{align*}
If $n$  denotes the order of $q$, then 
\[
q^k=1
\quad\Longleftrightarrow\quad
n\mid k
\quad\Longleftrightarrow\quad
k\equiv 0 \pmod n.
\]
\end{proof}
More generally, the \textbf{$q$-factorial} is defined as
$$[k]_q! = [k]_q [k-1]_q \cdots [1]_q.$$
Using the above, the \textbf{Gaussian binomial coefficients} are defined as follows:
\[
\binom{k}{i}_q=\frac{[k]_q!}{[i]_q!\,[k-i]_q!}.
\]
Furthermore, if $f\in \K[t]$is a polynomial of degree $d$ written explicitly as
\[
f(t)=\sum_{j=0}^{d} c_j\,t^{j}\qquad (c_j\in \K,\; c_d\neq 0),
\]
the \textbf{support} of $f$ consists of
\[
\operatorname{supp}(f)=\{\, j \mid c_j\neq 0 \,\}.
\]
The identities below will be useful in the proof of the PI property for the algebra $B_q(f)$.

\begin{lema}{\cite[Lemma 3.1]{GaddisYee2025}}\label{lema3.1}
Let $\K$ be a field and $q\in \K^{*}$. Suppose that $f\in \K[t]$has degree $d$ and
\[
f(t)=\sum_{j=0}^{d}c_j t^j\qquad (c_j\in \K,\;c_d\neq 0).
\]
Then, for every integer $k\ge 1$, the following identities hold
\begin{align}
\delta(u^k) &= \sum_{j=0}^{d} [k]_{q^{\,j+1}}\; c_j\, v^j\, u^{k-1},
\label{eq:delta-uk}\\[2mm]
\delta(v^k) &= \sum_{j=0}^{d} [k]_{q^{-(j+1)}}\; c_j\, u^j\, v^{k-1}, \label{eq:delta-vk}
\end{align}
where $\delta$ is the $\sigma$-derivation defined in Lemma \ref{B_q(f)ore}.
\end{lema}
\begin{proof}
Recall that
\[
\begin{aligned}
\sigma(u) &= q\,u, &\qquad \sigma(v) &= q^{-1}v,\\
\delta(u) &= f(v)=\sum_{j=0}^{d}c_j v^j, &\qquad 
\delta(v) &= f(u)=\sum_{j=0}^{d}c_j u^j.
\end{aligned}
\]
Moreover, for $j\ge 0$ one has 
\begin{equation}\label{eq:comm-uvj}
u\,v^j=q^j v^j u \qquad  \textit{and} \qquad  v\,u^j=q^{-j}u^j v. 
\end{equation}
The proof proceeds by induction on $k\ge 1$. For $k=1$, since $[1]_q=1$ for every $q\in \K^{*}$, one obtains  
\begin{align*}
\sum_{j=0}^{d}[1]_{q^{j+1}}\,c_j\,v^j\,u^{1-1}
&=\sum_{j=0}^{d}c_j\,v^j=f(v)=\delta(u),\\
\sum_{j=0}^{d}[1]_{q^{-(j+1)}}\,c_j\,u^j\,v^{1-1}
&=\sum_{j=0}^{d}c_j\,u^j
=f(u)
=\delta(v).
\end{align*}
Now, suppose that identity  \eqref{eq:delta-uk} holds for $k$. Then,
\[
\begin{aligned}
\delta(u^{k+1})
&=\delta(u\,u^k)\\
&=\sigma(u)\,\delta(u^k)+\delta(u)\,u^k\\
&=qu\,\delta(u^k)+f(v)\,u^k\\
&=qu\left(\sum_{j=0}^{d}[k]_{q^{j+1}}\,c_j\,v^j\,u^{k-1}\right)
   +\left(\sum_{j=0}^{d}c_j\,v^j\right)u^k \\
&=\sum_{j=0}^{d}[k]_{q^{j+1}}\,c_j\,q(u\,v^j)\,u^{k-1}
   +\sum_{j=0}^{d}c_j\,v^j\,u^k \\
&=\sum_{j=0}^{d}[k]_{q^{j+1}}\,c_j\,q\,(q^j v^j u)\,u^{k-1}
   +\sum_{j=0}^{d}c_j\,v^j\,u^k\\ 
&=\sum_{j=0}^{d}\bigl(q^{j+1}[k]_{q^{j+1}}+1\bigr)c_j\,v^j\,u^k \\ 
&=\sum_{j=0}^{d}[k+1]_{q^{j+1}}\,c_j\,v^j\,u^k.
\end{aligned}
\]
Similarly, suppose that identity \eqref{eq:delta-vk} holds for $k$. Thus,
\[
\begin{aligned}
\delta(v^{k+1})
&=\delta(v\,v^k)\\
&=\sigma(v)\,\delta(v^k)+\delta(v)\,v^k\\
&=q^{-1}v\,\delta(v^k)+f(u)\,v^k\\
&=q^{-1}v\left(\sum_{j=0}^{d}[k]_{q^{-(j+1)}}\,c_j\,u^j\,v^{k-1}\right)
   +\left(\sum_{j=0}^{d}c_j\,u^j\right)v^k\\
&=\sum_{j=0}^{d}[k]_{q^{-(j+1)}}\,c_j\,(q^{-1}v\,u^j)\,v^{k-1}
   +\sum_{j=0}^{d}c_j\,u^j\,v^k \\
&=\sum_{j=0}^{d}[k]_{q^{-(j+1)}}\,c_j\,q^{-1}\,(q^{-j}u^j v)\,v^{k-1}
   +\sum_{j=0}^{d}c_j\,u^j\,v^k\\
&=\sum_{j=0}^{d}\bigl(q^{-(j+1)}[k]_{q^{-(j+1)}}+1\bigr)c_j\,u^j\,v^k\\[2mm]
&=\sum_{j=0}^{d}[k+1]_{q^{-(j+1)}}\,c_j\,u^j\,v^k.
\end{aligned}
\]
\end{proof}
\begin{lema}{\cite[Lemma 3.2]{GaddisYee2025}}\label{lema3.2}
Let $\K$ be a field and  $q\in \K^{*}$. Suppose that $f\in \K[t]$ has degree $d$. 
Then, for every integer $k\ge 1$, one has
\begin{align}
w\,u^{k} &= q^{k}\,u^{k}\,w \;+\; \sum_{j=0}^{d}[k]_{q^{j+1}}\,c_j\,v^{j}\,u^{k-1}, 
\label{eq:delta-uk3}\\[2mm]
w\,v^{k} &= q^{-k}\,v^{k}\,w \;+\; \sum_{j=0}^{d}[k]_{q^{-(j+1)}}\,c_j\,u^{j}\,v^{k-1}.
\label{eq:delta-uk4}
\end{align}
\end{lema}
\begin{proof}
Remember that $\sigma(u)=q\,u$; hence,  \[
\sigma(u^k)=\sigma(\underbrace{u\cdots u}_{k\text{ times}})
=\underbrace{\sigma(u)\cdots \sigma(u)}_{k\text{ times }}
=(qu)^k=q^k u^k.
\]
By Lemma \ref{B_q(f)ore}, the following relation holds in  $B_q(f)$: 
\begin{equation}\label{eq:Ore-rule}
w\,a=\sigma(a)\,w+\delta(a)\qquad \text{for all }a\in K_q[u,v].
\end{equation}
Setting $a=u^k$ in \eqref{eq:Ore-rule}, one obtains
\[
w\,u^k=\sigma(u^k)\,w+\delta(u^k)=q^k u^k w+\delta(u^k).
\]
By identity  \eqref{eq:delta-uk} of Lemma \ref{lema3.1},
\[
\delta(u^k)=\sum_{j=0}^{d}[k]_{q^{j+1}}\,c_j\,v^j\,u^{k-1}.
\]
Substituting this into the preceding equality, one concludes that
\[
w\,u^k=q^k u^k w+\sum_{j=0}^{d}[k]_{q^{j+1}}\,c_j\,v^j\,u^{k-1}.
\]
The proof of identity (\ref{eq:delta-uk4}) is obtained by arguments analogous to those above.
\end{proof}

\begin{lema}{\cite[Lemma 3.3]{GaddisYee2025}}\label{lema3.3}
Let $q$ be a primitive $n$-th root of unity, with $n\geq 2$. Suppose that $n\nmid (j+1)$ for every  $j \in \operatorname{supp}(f)$.
Then $u^{n}$, $v^{n}\in Z(B_q(f))$.
\end{lema}
\begin{proof}
First of all, the element $u^n$ commutes with $v$ and $w$. Indeed,  recall that $v u^n=q^{-n}u^n v$. Since $q^{-n}=1$, it follows that $v u^n=u^n v$. Additionally,  applying identity (\ref{eq:delta-uk3}) of Lemma \ref{lema3.2} with $k=n$, one has
\[
w u^n \;=\; q^n u^n w \;+\;\sum_{j=0}^{d}[n]_{q^{j+1}}\,c_j\,v^j\,u^{n-1}.
\]
If  $j\notin \operatorname{supp}(f)$, then $c_j=0$  and the corresponding summand vanishes. If $j\in \operatorname{supp}(f)$, then  $c_j\neq 0$ and by hypothesis $n\nmid (j+1)$, so  in particular $q^{j+1}\neq 1$.  However, $(q^{j+1})^n=(q^n)^{j+1}=1$, so $q^{j+1}$ is an $n$-th root of unity distinct from  $1$. Thus, 
\[
[n]_{q^{j+1}}=\frac{(q^{j+1})^n-1}{q^{j+1}-1}=\frac{1-1}{q^{j+1}-1}=0.
\]
Consequently, $w u^n=u^n w$, that is $u^n$ commutes with $w$. Therefore $u^n \in Z\!\big(B_q(f)\big)$. The proof for $v^n$ is obtained by arguments analogous to those above.

\end{proof}

The hypothesis $n\nmid (j+1)$  for every $j\in \operatorname{supp}(f)$ is essential. For instance,
if $f(t)=t^8$ and $q^3=1$, then $n=3$ and $\operatorname{supp}(f)=\{8\}$, but
\[
n\mid(8+1)=9,
\]
so $q^{8+1}=q^9=(q^3)^3=1$. In this case, identity \eqref{eq:delta-uk3} does not yield the desired vanishing, since $[k]_{q^{8+1}}=[k]_{1} =k$ 
and therefore 
\[
\sum_{j=0}^{8}[k]_{q^{j+1}}\,c_j\,v^j\,u^{k-1}
=
[k]_{q^{8+1}}\,c_8\,v^8\,u^{k-1}
=
[k]_{q^{9}}\,v^8\,u^{k-1}= kv^8\,u^{k-1}
\]
Since the base field $\K$ is assumed to have $\operatorname{char}(\K)=0$,
then $k$ is non-zero in $\K$, for every $k\in \Z^+$. In particular
the coefficient $k$ does not vanish  and the term 
$k\,v^8u^{k-1}$ is non-zero. So the conclusions of the two preceding lemmas need not hold.


\begin{lema}\label{wnidentidad}
For every $k\ge 1$ and polynomial $f\in \K[t]$, the following holds
\begin{equation}\label{eq:3.7}
w^{k}u \;=\; q^{k}\,u\,w^{k}\;+\;\sum_{i=0}^{k-1} w^{\,k-1-i}\, f(v)\, (q w)^{i}.
\end{equation}
\end{lema}
\begin{proof}
Identity   \eqref{eq:3.7} is established by induction on $k\ge 1$. For $k=1$ the identity reduces to
\[
w^1u = q^1u w^1 + \sum_{i=0}^{0} w^{0}f(v)(qw)^0
= quw + f(v),
\] 
which coincides with one of the defining relations of the algebra $B_q(f)$. Suppose that identity \eqref{eq:3.7} holds for $k$. Multiplying on the left by $w$ and distributing, one obtains
\[
w^{k+1}u = w(w^ku)
= q^{k}\, w u w^{k} + \sum_{i=0}^{k-1} w^{k-i} f(v)(qw)^i.
\]
Applying the case $k=1$ to replace $wu$, one has
\[
wuw^k = (quw+f(v))w^k = quw^{k+1}+f(v)w^k.
\]
Therefore,
\[
\begin{aligned}
w^{k+1}u
&= q^{k}\bigl(quw^{k+1}+f(v)w^k\bigr)
  + \sum_{i=0}^{k-1} w^{k-i} f(v)(qw)^i \\
&= q^{k+1}u w^{k+1} + q^{k} f(v) w^k
  + \sum_{i=0}^{k-1} w^{k-i} f(v)(qw)^i\\
&= q^{k+1}u w^{k+1} + \sum_{i=0}^{k} w^{k-i} f(v)(qw)^i. 
\end{aligned}
\]
which completes the induction.
\end{proof}

\begin{lema}{\cite[Lemma 3.7]{GaddisYee2025}}\label{lema3.7GY}
Suppose that $q$ is a primitive $n$-th root of unity with $n\geq 2$,
and let $f=\sum_{j=0}^{d} c_j t^j\in \K[t]$, with
$\operatorname{supp}(f)=\{\,j\mid c_j\neq 0\,\}$.
Then $f(u)$, $f(v)\in Z(B_q(f))$ if and only if \;\;$n\mid j$ for every  
$j \in \operatorname{supp}(f)$ 
In particular, if $n\mid j$ for every $j\in \operatorname{supp}(f)$, then
$w^n\in Z(B_q(f))$.
\end{lema}
\begin{proof}
Suppose that $f(u),f(v)\in Z(B_q(f))$.
In particular, $f(u)v=vf(u)$. Using the fact that $u^jv=q^jvu^j$ for every $j\ge 0$, one computes,
\[
\begin{aligned}
f(u)v
&=\left(\sum_{j=0}^{d}c_j u^j\right)v
=\sum_{j=0}^{d}c_j u^jv
=\sum_{j=0}^{d}c_j q^j v u^j\\
&= v\left(\sum_{j=0}^{d}c_j (qu)^j\right)
= v\,f(qu).
\end{aligned}
\]
Therefore,  $vf(u)=vf(qu)$ and  $v\bigl(f(u)-f(qu)\bigr)=0$. Since $B_q(f)$ is a domain, it follows that $f(u)=f(qu)$; that is,
$\sum_{j=0}^{d}c_j u^j=\sum_{j=0}^{d}c_j q^j u^j$ from which it follows that 
$\sum_{j=0}^{d}c_j(1-q^j)u^j=0$.
However,  the powers  $1,u,u^2,\dots$ are linearly independent, then
$c_j(1-q^j)=0$ for every $j$. In fact,  if $j\in\operatorname{supp}(f)$ one obtains $q^j=1$.
Since $q$ is a primitive root of order $n$,  this is equivalent to  $n\mid j$.\\

Conversely, suppose that $n\mid j$, for every $j\in\operatorname{supp}(f)$. Thus, 
\[
\begin{aligned}
f(qu)
&=\sum_{j=0}^{d}c_j (qu)^j
=\sum_{j=0}^{d}c_j q^j u^j
=f(u),\\
f(q^{-1}v)
&=\sum_{j=0}^{d}c_j (q^{-1}v)^j
=\sum_{j=0}^{d}c_j q^{-j}v^j
=f(v).
\end{aligned}
\]
It follows that
\[
\begin{aligned}
f(u)v &= v f(qu)= v f(u),\\
u f(v) &= f(q^{-1}v)\,u = f(v)\,u;
\end{aligned}
\]
that is, $f(u)$  commutes with $v$ and $f(v)$ commutes with $u$. 
To verify that both elements also commute with $w$,  nota that
$\delta(f(u))=\sum_{j=0}^{d}c_j\,\delta(u^j)$. Applying identity  (\ref{eq:delta-uk}) of Lemma \ref{lema3.1},
\[
\delta(u^j)=\left(\sum_{i=0}^{d}[j]_{q^{i+1}}\,c_i\,v^i\right)u^{j-1},
\]
so that
\[
\delta(f(u))=\sum_{j=0}^{d}c_j\left(\sum_{i=0}^{d}[j]_{q^{i+1}}\,c_i\,v^i\right)u^{j-1}.
\]
If $c_i\neq 0$, then $i\in\operatorname{supp}(f)$ and  $n\mid i$. Because $n\geq 2$, one has $n\nmid (i+1)$ and  $q^{i+1}\neq 1$.
Moreover, if $c_j\neq 0$, then  $n\mid j$ and so  $(q^{i+1})^j=q^{(i+1)j}=1$.
Hence,
\[
[j]_{q^{i+1}}=\frac{(q^{i+1})^j-1}{q^{i+1}-1}=0.
\]
One concludes that $\delta(f(u))=0$, and by analogous arguments the equality $\delta(f(v))=0$ is obtained. Thus,
\[
wf(u)=\sigma(f(u))w+\delta(f(u))=f(qu)w=f(u)w.
\]
Likewise $wf(v)=f(v)w$ and, in consequence, $f(u),f(v)\in Z(B_q(f))$.

It remains to show that $w^n$ is a central element. Applying Lemma \ref{wnidentidad}, 
\[
w^n u=q^n u w^n+\sum_{i=0}^{n-1}w^{n-1-i}f(v)(qw)^i.
\]
Indeed, 
\[
\begin{aligned}
\sum_{i=0}^{n-1}w^{n-1-i}f(v)(qw)^i
&=\sum_{i=0}^{n-1}q^i\,f(v)\,w^{n-1}
=\left(\sum_{i=0}^{n-1}q^i\right)f(v)\,w^{n-1}\\
&=[n]_q\,f(v)\,w^{n-1},
\end{aligned}
\]
so $w^nu=q^nuw^n+[n]_q\,f(v)\,w^{n-1}$.
Since $q$ is a primitive $n$-th root of unity it follows that $[n]_q=0$ and $w^n u=uw^n$.
By analogous arguments identity  $w^n v=vw^n$ is proved, and therefore $w^n\in Z(B_q(f))$.
\end{proof}


\begin{proposicion}\label{propiedadpindividej}
Suppose that $q$ is a primitive 
 $n$-th  root of unity with $n\geq 2$, and consider the algebra $B_q(f)$. If $n\mid j$, for every $j\in \operatorname{supp}(f)$, then  $B_q(f)$ is a PI algebra.
\end{proposicion}
\begin{proof}
By the particular case of Lemma~\ref{lema3.7GY}, one has $w^n\in Z\!\big(B_q(f)\big)$.
Moreover, since $n\mid j$ for every $j\in \operatorname{supp}(f)$
and $n\geq 2$, it follows that $n\nmid (j+1)$ for every
$j\in \operatorname{supp}(f)$; otherwise $n$ would divide
$(j+1)-j=1$. Hence the hypotheses of Lemma~\ref{lema3.3}
are satisfied, and therefore $u^n$ and $v^n$ are central elements
of $B_q(f)$.

By Proposition \ref{basepbw},  every element of
$B_q(f)$ is a finite linear combination of monomials $u^a v^b w^c$. Since 
$u^n$, $v^n$ and $w^n$ are central, for any $a,b,c\in\mathbb{N}$
one may write
\[
a= n\alpha + r,\quad b=n\beta + s,\quad c=n\gamma + t,
\qquad 0\le r,s,t \le n-1.
\]
Thus,
\[
u^a v^b w^c
=(u^n)^{\alpha}(v^n)^{\beta}(w^n)^{\gamma}\,u^r v^s w^t,
\]
where the factor 
 $(u^n)^{\alpha}(v^n)^{\beta}(w^n)^{\gamma}$ belongs to
$Z\!\big(B_q(f)\big)$. Consequently, the algebra $B_q(f)$ is generated as a left
$Z\!\big(B_q(f)\big)$-module  by the finite set
\[
\{\,u^r v^s w^t \mid 0\le r,s,t\le n-1\,\}.
\]
So $B_q(f)$ is a finitely generated $Z\!\big(B_q(f)\big)$-module  and \cite[\S 13.1.13, Corollary]{mcconnell} implies that  $B_q(f)$  satisfies the PI property.
\end{proof}

\begin{ejemplo}
Fix an integer $n\geq 2$ and let $q\in \K$ be a primitive $n$-th root of unity. Consider the following examples of polynomials $f\in \K[t]$:
\begin{itemize}
\item $f(t)=a_0+a_1t^{n}+a_2t^{3n}$, with $a_0,a_1,a_2\in \K$ and $a_1\neq 0$.
In this case $\operatorname{supp}(f)\subseteq\{0,n,3n\}$, thus
 $n\mid j$, for every $j\in \operatorname{supp}(f)$. By Proposition
\ref{propiedadpindividej}, the algebra $B_q(f)$ satisfies the PI property.
\item $f(t)=b_1t^{2n}+b_2t^{4n}+b_3t^{5n}$, with $b_1,b_2,b_3\in \K$ and $b_1\neq 0$.
Then $\operatorname{supp}(f)\subseteq\{2n,4n,5n\}$. In particular, $n\mid j$ for every $j\in \operatorname{supp}(f)$ and Proposition \ref{propiedadpindividej} ensures that $B_q(f)$ is a PI algebra.
\end{itemize}
\end{ejemplo}

Although the proof of the preceding result is particularly elegant, drawing on the various identities and properties developed throughout this work, it does not guarantee that $B_q(f)$ satisfies the PI property in all cases. This is illustrated by the example below.

\begin{ejemplo}
Take $n=4$ and let $q\in \K$ be a primitive fourth root of unity. Consider $f(t)=t+t^{5}\in \K[t]$. Then $\operatorname{supp}(f)=\{1,5\}$. 
For each $j\in\{1,5\}$ one has $4\nmid j$, so the hypothesis of Proposition \ref{propiedadpindividej} is not satisfied. However, $4\nmid (j+1)$  for each $j\in\{1,5\}$, so the hypothesis of Lemma \ref{lema3.3} holds, and therefore $u^4$, $v^4\in Z(B_q(f))$.
\end{ejemplo}

The following result, established by Gaddis and Yee in \cite{GaddisYee2025}, generalizes the preceding proposition.

\begin{proposicion}{\cite[Proposition 3.11]{GaddisYee2025}} \label{prop:GY311}
Let $q$ be a primitive $n$-th root of unity with $n\geq 2$ and suppose that $n\nmid (j+1)$ for every $j\in \operatorname{supp}(f)$. Then $B_q(f)$ is a PI algebra.
\end{proposicion}
\begin{proof}
Let $B:=B_q(f)$ and $Z:=Z(B)$. 
Lemma \ref{lema3.3} implies that $u^n,v^n \in Z$. In particular, $\K[u^n,v^n]\subseteq Z$ and the following holds
\[
\operatorname{GKdim}(Z)\ \ge\ \operatorname{GKdim}\big(\K[u^n,v^n]\big)=2.
\]
Since $B$ is a domain and $Z\setminus\{0\}$ is a central multiplicative
set of regular elements of $B$, the localization $\widehat{B}:=BZ^{-1}$
is well defined \cite[Lemma~2.1.3]{mcconnell}. 
Now, suppose that $B$ does not satisfy the PI property. Since  $B$ is finitely generated as $\K$-algebra, it is not locally PI\footnote{An algebra $A$ is locally PI if every finitely generated subalgebra $C\subseteq A$ satisfies the PI property.}:  indeed, $B$ itself is a finitely generated subalgebra of itself. Consequently, $\widehat{B}$ is not locally PI either, for if it were, then every finitely generated subalgebra of $B$, regarded as a subalgebra of $\widehat{B}$, would satisfy the PI property, and $B$ would therefore be locally PI, contradicting the assumption.
Applying Corollary 2 in \cite{SmithZhang1998}, one obtains
\[
\operatorname{GKdim}(B)\ \ge\ 2+\operatorname{GKdim}(Z).
\]
Moreover, Lemma 2.5 in \cite{GaddisYee2025} implies that $\operatorname{GKdim}(B)=3$, thus
\[
3=\operatorname{GKdim}(B)\ \ge\ 2+\operatorname{GKdim}(Z)\ \ge\ 2+2=4,
\]
a contradiction. Hence $B$ satisfies the PI property.
\end{proof}

When $q$ is not a root of unity, the following result is an immediate consequence.

\begin{proposicion}
If $q$ is not a root of unity, then $B_q(f)$ is not a PI algebra.
\end{proposicion}
\begin{proof}
If $B_q(f)$ were a PI algebra, then every subalgebra would also satisfy the PI property; in particular, the quantum plane $\K_q[u,v]\subseteq B_q(f)$ would be PI. However, it is well known, by \cite[Section \S 7.1]{DeConciniProcesi1993}, that $\K_q[u,v]$ is PI if and only if $q$ is a root of unity. Therefore, if $q$ is not a root of unity, $B_q(f)$ cannot be a PI algebra. 
\end{proof} 
\subsection*{Central elements of \texorpdfstring{$B_q(f)$}{Bq(f)} in positive characteristic}
The result below shows how the characteristic of the base field directly influences the classification of certain central elements of some algebras. For $q=1$, one defines $[k]_1:=k\cdot 1_{\K}$. Thus $[k]_1=0$ if and only if $\operatorname{char}(\K)=p>0$ and $p\mid k$. 
\begin{lema}[A positive-characteristic version of Lemma~\ref{lema3.3}]
Let $\K$  be a field with $\operatorname{char}(\K)=p>0$ and let  $q$ be a primitive 
$n$-th root of unity. Suppose that one of the following conditions holds:
\begin{enumerate}
\item[\rm (i)] $n\nmid (j+1)$ for all $j\in \operatorname{supp}(f)$; or
\item[\rm (ii)] $p\mid n$.
\end{enumerate}
Then
\[
u^{n},\,v^{n}\in Z\!\big(B_q(f)\big).
\]
\end{lema}
\begin{proof}
Since $q^n=1$, one has $vu^n=q^{-n}u^n v=u^n v$, so $u^n$ commutes with $v$. Moreover, applying identity (\ref{eq:delta-uk3}) of Lemma \ref{lema3.2} with $k=n$ 
\[
w u^n=q^n u^n w+\sum_{j=0}^{d}[n]_{q^{j+1}}\,c_j\,v^j\,u^{n-1}
=u^n w+\sum_{j=0}^{d}[n]_{q^{j+1}}\,c_j\,v^j\,u^{n-1}.
\]
It therefore suffices to show that the sum vanishes. If $j\notin \operatorname{supp}(f)$, then $c_j=0$ and the corresponding summand is zero. Suppose 
 $j\in \operatorname{supp}(f)$, so that $c_j\neq 0$.
Since $q$ has order $n$, 
\[
q^{j+1}=1 \quad\Longleftrightarrow\quad n\mid (j+1).
\]
If condition (i) holds, then  $n\nmid (j+1)$, hence $q^{j+1}\neq 1$ and $(q^{j+1})^n=1$.
 Consequently,
\[
[n]_{q^{j+1}}=\frac{(q^{j+1})^n-1}{q^{j+1}-1}=\frac{1-1}{q^{j+1}-1}=0.
\]
On the other hand, if condition  (ii) is satisfied then $n=0$ as an element of $\K$ and
\[
[n]_1=\underbrace{1+\cdots+1}_{n\ \text{times}}=n=0\quad \text{in }\K.
\]
Thus, $[n]_{q^{j+1}}=[n]_1=0$.

In either case,  $[n]_{q^{j+1}}=0$, for every $j\in \operatorname{supp}(f)$, 
and the sum vanishes. It follows that $w u^n=u^n w$, that is, $u^n\in Z(B_q(f))$. Finally, the proof that $v^n\in  Z(B_q(f))$ is analogous, applying identity (\ref{eq:delta-uk4}) of Lemma \ref{lema3.2}.
\end{proof}

\section{Final comments and open questions}

In this section, we present some final comments and open questions that may
inspire future research.

One possible direction for future work is to study the PI property for other
classes of algebras. For instance, one may consider other double Ore
extensions, in particular the trimmed double Ore extensions defined in
Section~\ref{oredobles}. It is also natural to consider those double Ore
extensions that can be realized as iterated Ore extensions; in this case, the
results developed in \cite{LeroyMatczuk2007} could be applied.

Let $f\in \K[t]$ be a polynomial, and fix scalars $r,s,\gamma\in \K$. The
\emph{generalized Down--Up algebra} $L=L(f,r,s,\gamma)$ is the unital
a\-sso\-ciative $k$-algebra generated by $d,u,h$, subject to the relations
\begin{align}
dh-rhd+\gamma d&=0, \tag{1.1}\\
hu-ruh+\gamma u&=0, \tag{1.2}\\
du-sud+f(h)&=0. \tag{1.3}
\end{align}

When $f$ has degree one, one recovers all down--up algebras
$A(\alpha,\beta,\gamma)$, with $\alpha,\beta,\gamma\in \K$, for suitable
choices of the parameters defining $L$. In this context, it is natural to
study the PI pro\-per\-ty for this class of algebras. In particular, one may ask
under which conditions on the parameters $r,s,\gamma$ and on the polynomial
$f$ the algebra $L(f,r,s,\gamma)$ satisfies a polynomial identity.

On the other hand, let $R$ and $A$ be rings. We say that $A$ is a
\emph{skew PBW extension} of $R$, or simply a \emph{$\sigma$-PBW extension},
if the following conditions hold:
\begin{itemize}
\item[{\rm (i)}] $R\subseteq A$.

\item[{\rm (ii)}] There exist elements $x_1,x_2,\dots,x_n\in A$ such that
$A$ is a left free $R$-module with basis $\operatorname{Mon}(A)$, where
$\operatorname{Mon}(A)$ denotes the set of standard monomials; that is,
\[
\operatorname{Mon}(A):=
\left\{
x_1^{\alpha_1}\cdots x_n^{\alpha_n}
\mid
\alpha=(\alpha_1,\dots,\alpha_n)\in \mathbb{N}^n
\right\}.
\]

\item[{\rm (iii)}] For each $1\leq i\leq n$ and each
$r\in R\setminus\{0\}$, there exists $c_{i,r}\in R\setminus\{0\}$ such that
\[
x_ir-c_{i,r}x_i\in R.
\]

\item[{\rm (iv)}] For all $1\leq i,j\leq n$, there exists
$c_{i,j}\in R\setminus\{0\}$ such that
\[
x_jx_i-c_{i,j}x_ix_j\in R+Rx_1+\cdots+Rx_n.
\]
\end{itemize}
Skew PBW extensions are denoted by
\[
A:=\sigma(R)\langle x_1,\dots,x_n\rangle.
\]

From this definition, another interesting direction is to study the PI
property for skew PBW extensions, at least for some particular subclasses of
these algebras. One motivation for this is that several algebras considered
in this paper, such as the two-parameter quantum Heisenberg algebras and the
algebras $B_q(f)$, can be realized as skew PBW extensions in suitable
presentations.

\subsection*{The Ozone group}

Let $A$ be a $\K$-algebra and let $Z=Z(A)$ be its center. The
\emph{Ozone group} of $A$ over $Z$ is defined as the group of automorphisms
of $A$ that fix the center pointwise; that is,
\[
\operatorname{Oz}(A):=\operatorname{Aut}_{Z\text{-alg}}(A).
\]

By \cite[Theorem~0.7]{ChanGaddisWonZhang2025}, if $A$ is a Noetherian
Artin--Schelter regular PI algebra and $Z=Z(A)$, then
\[
1\leq |\operatorname{Oz}(A)|\leq \operatorname{rk}_Z(A).
\]
This inequality shows that the order of the Ozone group is bounded above by
the rank of $A$ as a module over its center. For this reason, the explicit
computation of $\operatorname{Oz}(A)$, or at least the determination of
bounds for its order, constitutes another possible direction for future work.
Indeed, this perspective was one of the motivations for the development of
the present work.

\bibliographystyle{plain}
\bibliography{biblio}

@article{Launois2007,
  author = {Launois, S.},
  title = {Primitive ideals and automorphism group of {$U_q^+(B_2)$}},
  journal = {J. Algebra Appl.},
volume = {06},
number = {01},
pages = {21-47},
year = {2007},
doi = {10.1142/S0219498807002053},
URL = {https://doi.org/10.1142/S0219498807002053}
}

@article{Jimbo1985,
  author = {Jimbo, M.},
  title = {A q-difference analogue of \text{U}(g) and the \text{Y}ang-\text{B}axter equation},
  journal = {Lett. Math. Phys.},
  year = {1985},
  volume = {10},
  pages = {63--69}
}

@article{Drinfeld1986,
  author = {V. G. Drinfeld},
  title = {Quantum Groups},
  journal = {Proceedings of the International Congress of Mathematicians},
  year = {1986},
  pages = {798--820}
}

@article{benkart,
title = {Down–\text{U}p Algebras},
journal = {J. Algebra},
volume = {209},
number = {1},
pages = {305-344},
year = {1998},
issn = {0021-8693},
doi = {https://doi.org/10.1006/jabr.1998.7511},
url = {https://www.sciencedirect.com/science/article/pii/S0021869398975111},
author = {Benkart, G. and Roby, T.},
}

@misc{Bell,
  author       = {Bell, A. D.},
  title        = {Notes on localization in noncommutative noetherian rings},
  year         = {2000},
  howpublished = {Manuscript/notes},
  url          = {https://api.semanticscholar.org/CorpusID:162173433},
  note = {\url{https://api.semanticscholar.org/CorpusID:162173433}}
}

@article{Bavula4,
  author  = {Bavula, V. V. and Lenagan, T. H.},
  title   = {Generalized \text{W}eyl Algebras Are Tensor Krull Minimal},
  journal = {Journal of Algebra},
  volume  = {239},
  number  = {1},
  pages   = {93--111},
  year    = {2001},
  issn    = {0021-8693},
  doi     = {10.1006/jabr.2000.8641},
  url     = {https://www.sciencedirect.com/science/article/pii/S0021869300986411}
}

@article{Kirkman,
  author  = {Kirkman, E. E. and Musson, I. M. and Passman, D. S.},
  title   = {Noetherian \text{D}own-\text{U}p \text{A}lgebras},
  journal = {Proc. Amer. Math. Soc.},
  volume  = {127},
  number  = {11},
  pages   = {3161--3167},
  year    = {1999},
  url     = {http://www.jstor.org/stable/119509},
  note    = {JSTOR}
}

@article{carvalho,
  author  = {Carvalho, P. A. A. B. and Musson, Ian M.},
  title   = {Down--\text{U}p Algebras and Their Representation Theory},
  journal = {J. Algebra},
  volume  = {228},
  number  = {1},
  pages   = {286--310},
  year    = {2000},
  issn    = {0021-8693},
  doi     = {10.1006/jabr.1999.8263},
  url     = {https://www.sciencedirect.com/science/article/pii/S0021869399982637}
}

@article{carvalho2,
  author  = {Carvalho, P. A. A. B. and Musson, I. M.},
  title   = {Monolithic Modules Over Noetherian Rings},
  journal = {Glasg. Math. J.},
  volume  = {53},
  pages   = {483--491},
  year    = {2011},
  doi     = {10.1017/S001708951000085X}
}

@phdthesis{carvalhotesis,
  author  = {Carvalho, P. A. A. B.},
  title   = {Some Noetherian Rings},
  school  = {University of Glasgow},
  address = {Glasgow, Scotland},
  year    = {1997},
  url     = {https://theses.gla.ac.uk},
  note    = {Available at University of Glasgow Theses Service}
}

@article{carvalho3,
  author  = {Carvalho, P. A. A. B. and Lomp, C. and Pusat-Yilmaz, D.},
  title   = {Injective modules over \text{D}own–\text{U}p algebras},
  journal = {Glasg. Math. J.},
  volume  = {52},
  pages   = {53--59},
  year    = {2010},
  doi     = {10.1017/S0017089510000261}
}

@article {dehn,
    AUTHOR = {Dehn, M.},
     TITLE = {{\"U}ber die {G}rundlagen der projektiven {G}eometrie und
              allgemeine {Z}ahlsysteme},
   JOURNAL = {Math. Ann.},
  FJOURNAL = {Mathematische Annalen},
    VOLUME = {85},
      YEAR = {1922},
    NUMBER = {1},
     PAGES = {184--194},
      ISSN = {0025-5831,1432-1807},
   MRCLASS = {99-04},
  MRNUMBER = {1512061},
       DOI = {10.1007/BF01449618},
       URL = {https://doi.org/10.1007/BF01449618},
}

@article{gallego,
  author  = {Gallego, C. and Solotar, A.},
  title   = {Stable rank of \text{D}own-\text{U}p algebras},
  journal = {J. Algebra},
  volume  = {526},
  pages   = {266--282},
  year    = {2019},
  issn    = {0021-8693},
  doi     = {10.1016/j.jalgebra.2018.02.037},
  url     = {https://www.sciencedirect.com/science/article/pii/S0021869318301959}
}

@article{Jordan,
  author  = {Jordan, D. A.},
  title   = {\text{D}own--\text{U}p Algebras and Ambiskew Polynomial Rings},
  journal = {J. Algebra},
  volume  = {228},
  number  = {1},
  pages   = {311--346},
  year    = {2000},
  issn    = {0021-8693},
  doi     = {10.1006/jabr.1999.8264},
  url     = {https://www.sciencedirect.com/science/article/pii/S0021869399982649}
}

@book {jacobson,
    AUTHOR = {Jacobson, N.},
     TITLE = {Collected mathematical papers. {V}ol. 2},
    SERIES = {Contemporary Mathematicians},
      NOTE = {(1947--1965)},
 PUBLISHER = {Birkh\"auser Boston, Inc., Boston, MA},
      YEAR = {1989},
     PAGES = {xviii+556},
      ISBN = {0-8176-3411-8},
   MRCLASS = {01A75 (16-03)},
  MRNUMBER = {1036025},
MRREVIEWER = {S.\ S.\ Page},
       DOI = {10.1007/978-1-4612-3692-4},
       URL = {https://doi.org/10.1007/978-1-4612-3692-4},
}

@article {kaplansky,
    AUTHOR = {Kaplansky, I.},
     TITLE = {Rings with a polynomial identity},
   JOURNAL = {Bull. Amer. Math. Soc.},
  FJOURNAL = {Bulletin of the American Mathematical Society},
    VOLUME = {54},
      YEAR = {1948},
     PAGES = {575--580},
      ISSN = {0002-9904},
   MRCLASS = {09.1X},
  MRNUMBER = {25451},
MRREVIEWER = {J.\ Dieudonn\'e},
       DOI = {10.1090/S0002-9904-1948-09049-8},
       URL = {https://doi.org/10.1090/S0002-9904-1948-09049-8},
}

@article{Kulkarni,
  author  = {Kulkarni, R. S.},
  title   = {Down--\text{U}p Algebras and Their Representations},
  journal = {J. Algebra},
  volume  = {245},
  number  = {2},
  pages   = {431--462},
  year    = {2001},
  issn    = {0021-8693},
  doi     = {10.1006/jabr.2001.8892},
  url     = {https://www.sciencedirect.com/science/article/pii/S0021869301988921}
}

@article{solotar,
  author  = {Chouhy, S. and Herscovich, E. and Solotar, A.},
  title   = {Hochschild homology and cohomology of \text{D}own–\text{U}p algebras},
  journal = {J. Algebra},
  volume  = {498},
  pages   = {102--128},
  year    = {2018},
  issn    = {0021-8693},
  doi     = {10.1016/j.jalgebra.2017.11.026},
  url     = {https://www.sciencedirect.com/science/article/pii/S0021869317306130}
}

@misc{Etingof2024,
  author       = {Etingof, P.},
  title        = {Lie \text{G}roups and \text{L}ie \text{A}lgebras \text{II}},
  howpublished = {MIT OpenCourseWare, Course 18.755},
  year         = {2024},
  note         = {Full lecture notes, Spring 2024},
  url          = {https://ocw.mit.edu/courses/18-755-lie-groups-and-lie-algebras-ii-spring-2024/mit18_755_s24_lec_full.pdf}
}

@book{mcconnell,
  author    = {McConnell, J. C. and Robson, J. C. and Small, L. W.},
  title     = {Noncommutative Noetherian Rings},
  publisher = {American Mathematical Society},
  series    = {Graduate Studies in Mathematics},
  volume    = {30},
  year      = {2001},
  isbn      = {9780821821695},
  url       = {https://books.google.com.co/books?id=DcTpBwAAQBAJ}
}

@incollection{AndruskiewitschDumas2008,
  author    = {Andruskiewitsch, N. and Dumas, F.},
  title     = {On the automorphisms of {$U_q^{+}(\mathfrak g)$}},
  booktitle = {Quantum Groups},
  series    = {IRMA Lectures in Mathematics and Theoretical Physics},
  volume    = {12},
  pages     = {107--133},
  publisher = {European Mathematical Society},
  address   = {Z{\"u}rich},
  year      = {2008}
}

@article{ArtinSchelterTate1991,
  author  = {Artin, M. and Schelter, W. and Tate, J.},
  title   = {Quantum deformations of {$\mathrm{GL}_n$}},
  journal = {Comm. Pure Appl. Math.},
  volume  = {44},
  number  = {8-9},
  pages   = {879--895},
  year    = {1991},
  doi     = {10.1002/cpa.3160440804}
}

@article{AkhavizadeganJordan1996,
  author  = {Akhavizadegan, M. and Jordan, D. A.},
  title   = {Prime ideals of quantized \text{W}eyl algebras},
  journal = {Glasg. Math. J.},
  volume  = {38},
  number  = {3},
  pages   = {283--297},
  year    = {1996},
  doi     = {10.1017/S0017089500031712}
}

@article{Bavula2023,
  author  = {Bavula, V. V.},
  title   = {Description of bi-quadratic algebras on 3 generators with {PBW} basis},
  journal = {J. Algebra},
  volume  = {631},
  pages   = {695--730},
  year    = {2023},
  doi     = {10.1016/j.jalgebra.2023.05.013},
  url     = {https://www.sciencedirect.com/science/article/pii/S0021869323002442}
}

@article{Bavula2024,
  author  = {Bavula, V. V.},
  title   = {The 3-cyclic quantum \text{W}eyl algebras, their prime spectra and a classification of simple modules ($q$ is not a root of unity)},
  journal = {J. Noncommut. Geom.},
  volume  = {18},
  pages   = {1041--1079},
  year    = {2024},
  doi     = {10.4171/JNCG/547},
  url     = {https://ems.press/journals/jncg/articles/14289645}
}

@misc{BeraMandalMukherjeeNandy2024,
  author       = {Bera, S. and Mandal, S. and Mukherjee, S. and Nandy, S.},
  title        = {Simple modules over 3-cyclic quantum \text{W}eyl algebra at roots of unity},
  year         = {2024},
  eprint       = {2406.14375},
  archivePrefix= {arXiv},
  primaryClass = {math.RT},
  url          = {https://arxiv.org/abs/2406.14375},
  note         = {\url{https://arxiv.org/abs/2406.14375}}
}

@article{BeraMukherjee2023,
  author  = {Bera, S. and Mukherjee, S.},
  title   = {\text{PI} degree of quantized \text{W}eyl algebras},
  journal = {Proc. Indian Acad. Sci. Math. Sci.},
  volume  = {133},
  pages   = {18},
  year    = {2023}
}

@misc{BeraMukherjee2025,
  author       = {Bera, S and Mukherjee, S.},
  title        = {On simple modules over the quantum matrix algebra at roots of unity},
  year         = {2025},
  eprint       = {2503.10394},
  archivePrefix= {arXiv},
  primaryClass = {math.RT},
  url          = {https://arxiv.org/abs/2503.10394},
  note = {\url{https://arxiv.org/abs/2503.10394}}
}

@misc{BeraMukherjeeUqB2,
  author       = {Bera, S. and Mukherjee, S.},
  title        = {{$U_q^{+}(B_2)$ and its representations}},
  year         = {2025},
  eprint       = {2503.21170},
  archivePrefix= {arXiv},
  primaryClass = {math.RT},
  url          = {https://arxiv.org/abs/2503.21170},
  note = {\url{https://arxiv.org/abs/2503.21170}}
}

@phdthesis{cauchon,
  author  = {Cauchon, G.},
  title   = {Les T-anneaux et les anneaux {\`a} identit{\'e} polynomiale noeth{\'e}riens},
  school  = {Universit{\'e} de Paris-Sud},
  address = {Orsay},
  year    = {1977},
  type    = {Th{\`e}se}
}

@article{ChanGaddisWonZhang2025,
  author  = {Chan, K. and Gaddis, J. and Won, R. and Zhang, J. J.},
  title   = {\text{O}zone groups of \text{A}rtin--\text{S}chelter regular algebras satisfying a polynomial identity},
  journal = {Math. Z.},
  volume  = {310},
  pages   = {69},
  year    = {2025}
}

@article{damiano,
  author  = {Daniano, R. F. and Shapiro, J.},
  title   = {Twisted polynomial rings satisfying a polynomial identity},
  journal = {J. Algebra},
  volume  = {92},
  pages   = {116--127},
  year    = {1985}
}

@book{DrenskyFormanek2004,
  author    = {Drensky, V. and Formanek, E.},
  title     = {Polynomial Identity Rings},
  publisher = {Birkh{\"a}user},
  series    = {Advanced Courses in Mathematics -- CRM Barcelona},
  address   = {Basel},
  year      = {2004},
  isbn      = {978-3-7643-7126-5}
}

@article{Gaddis2016,
  author  = {Gaddis, J.},
  title   = {Two-parameter analogs of the \text{H}eisenberg enveloping algebra},
  journal = {Comm. Algebra},
  volume  = {44},
  number  = {11},
  pages   = {4637--4653},
  year    = {2016},
  doi     = {10.1080/00927872.2015.1101468},
  url     = {https://arxiv.org/abs/1308.4427}
}

@article{hayashi,
  author  = {Hayashi, T.},
  title   = {$q$-analogues of \text{C}lifford and \text{W}eyl algebras---spinor and oscillator representations of quantum enveloping algebras},
  journal = {Comm. Math. Phys.},
  volume  = {127},
  number  = {1},
  pages   = {129--144},
  year    = {1990},
  doi     = {10.1007/BF02096497},
  url     = {https://link.springer.com/article/10.1007/BF02096497}
}

@article{Jordan1995,
  author  = {Jordan, D. A.},
  title   = {A simple localization of the quantized \text{W}eyl algebra},
  journal = {J. Algebra},
  volume  = {174},
  number  = {1},
  pages   = {267--281},
  year    = {1995},
  doi     = {10.1006/jabr.1995.1128},
  url     = {https://www.sciencedirect.com/science/article/pii/S0021869385711283}
}

@article{ks,
  author  = {Kirkman, E. E. and Small, L. W.},
  title   = {$q$-analogs of harmonic oscillators and related rings},
  journal = {Israel J. Math.},
  volume  = {81},
  number  = {1-2},
  pages   = {111--127},
  year    = {1993}
}

@article{LeroyMatczuk2007,
  author  = {Leroy, A. and Matczuk, J.},
  title   = {Ore extensions satisfying a polynomial identity},
  journal = {J. Algebra Appl.},
  volume  = {2},
  number  = {3},
  pages   = {257--273},
  year    = {2007}
}

@article{Maltsiniotis1990,
  author  = {Maltsiniotis, G.},
  title   = {Groupes quantiques et structures diff{\'e}rentielles},
  journal = {C. R. Math. Acad. Sci. Paris},
  volume  = {311},
  pages   = {831--834},
  year    = {1990}
}

@article{pasca,
  author  = {Pascaud, J. and Valette, J.},
  title   = {Polyn{\^o}mes tordus {\`a} identit{\'e} polynomiale},
  journal = {Comm. Algebra},
  volume  = {16},
  pages   = {2415--2425},
  year    = {1988}
}

@article {wagner,
AUTHOR = {Wagner, W.},
TITLE = {{\"U}ber die {G}rundlagen der projektiven {G}eometrie und
allgemeine {Z}ahlensysteme},
JOURNAL = {Math. Ann.},
FJOURNAL = {Mathematische Annalen},
VOLUME = {113},
      YEAR = {1937},
    NUMBER = {1},
     PAGES = {528--567},
      ISSN = {0025-5831,1432-1807},
   MRCLASS = {99-04},
  MRNUMBER = {1513106},
       DOI = {10.1007/BF01571649},
       URL = {https://doi.org/10.1007/BF01571649},
}

@article{Takeuchi1990,
  author  = {Takeuchi, M.},
  title   = {A two-parameter quantization of $\mathrm{GL}_n$},
  journal = {Proceedings of the Japan Academy, Series A, Mathematical Sciences},
  volume  = {66},
  number  = {5},
  pages   = {112--114},
  year    = {1990}
}

@article{Zhang,
  author  = {Zhang, J. J. and Zhang, J.},
  title   = {Double \text{O}re extensions},
  journal = {J. Pure Appl. Algebra},
  volume  = {212},
  number  = {12},
  pages   = {2668--2690},
  year    = {2008},
  doi     = {10.1016/j.jpaa.2008.05.008},
  url     = {https://www.sciencedirect.com/science/article/pii/S0022404908001023}
}

@article{Zhang2,
  author  = {Zhang, J. J. and Zhang, J.},
  title   = {Double extension regular algebras of type (14641)},
  journal = {J. Algebra},
  volume  = {322},
  number  = {2},
  pages   = {373--409},
  year    = {2009},
  doi     = {10.1016/j.jalgebra.2009.03.041},
  url     = {https://www.sciencedirect.com/science/article/pii/S0021869309002270}
}

@misc{GaddisYee2025,
  author       = {Gaddis, J. and Yee, D.},
  title        = {A family of trivial \text{O}zone algebras},
  year         = {2025},
  eprint       = {2512.09766},
  archivePrefix= {arXiv},
  primaryClass = {math.RA},
  url          = {https://arxiv.org/abs/2512.09766},
  note ={\url{https://arxiv.org/abs/2512.09766}}
}

@article{SmithZhang1998,
  author  = {Smith, S. P. and Zhang, J. J.},
  title   = {A remark on {G}elfand--{K}irillov dimension},
  journal = {Proc. Amer. Math. Soc.},
  volume  = {126},
  number  = {2},
  pages   = {349--352},
  year    = {1998},
  doi     = {10.1090/S0002-9939-98-04074-X},
  url     = {https://www.ams.org/journals/proc/1998-126-02/S0002-9939-98-04074-X/}
}

@incollection{DeConciniProcesi1993,
  author    = {De Concini, C. and Procesi, C.},
  title     = {Quantum groups},
  booktitle = {$D$-modules, Representation Theory, and Quantum Groups},
  editor    = {Zampieri, Giuseppe and D'Agnolo, Andrea},
  series    = {Lecture Notes in Mathematics},
  volume    = {1565},
  pages     = {31--140},
  publisher = {Springer},
  address   = {Berlin},
  year      = {1993},
}
\bigskip
\noindent
J.G.:\\
Escuela de Matemáticas y Estadística, Universidad Pedagógica y Tecnológica de Colombia,\\
Tunja, Colombia\\
james.gomez@uptc.edu.co\\
\\
C.G.:\\
Departamento de Matemáticas, Pontificia Universidad Javeriana,\\
Bogotá, Colombia\\
gallegoj.cm@javeriana.edu.co

\end{document}